\documentclass{gtart}

\usepackage{amsthm, amsmath, amssymb, latexsym, amscd, graphicx, epstopdf}

\allowdisplaybreaks

\usepackage[english]{babel}
\usepackage[ansinew]{inputenc}
\usepackage{newtxtext,newtxmath}
\usepackage[scaled=0.91]{helvet}
\usepackage{fancyhdr}
\usepackage{indentfirst}
\usepackage{newlfont}
\usepackage{eucal}
\usepackage[all]{xy}
\usepackage[hypcap=true]{caption}
\usepackage{subcaption}
\usepackage{multirow}
\usepackage{lipsum}
\usepackage{tocloft}
\usepackage{cancel}
\usepackage{enumitem}
\usepackage{array}

\usepackage{tikz}
\usepackage{tikz-cd}
\usetikzlibrary{calc,matrix}
\usetikzlibrary{arrows}

\usepackage{titlesec}

\usepackage{hyperref}
\hypersetup{final=true}
\hypersetup{colorlinks=true}

\usepackage[inner=30mm, outer=30mm, textheight=225mm]{geometry}

\newcommand{\cyclic}[1]{\mathrel{ ( \mkern-3.85mu ( \mkern-6mu} {#1} \mathrel{\mkern-6mu ) \mkern-3.85mu)}}

\newcommand\restr[2]{{% we make the whole thing an ordinary symbol
		\left.\kern-\nulldelimiterspace % automatically resize the bar with \right
		#1 % the function
		\vphantom{\big|} % pretend it's a little taller at normal size
		\right|_{#2} % this is the delimiter
}}

\def\EE{\mathbb{E}} 
\def\CC{\mathbb{C}}
 
\def\HH{\mathbb{H}} 
\def\NN{\mathbb{N}} 
\def\PP{\mathbb{P}} 
\def\QQ{\mathbb{Q}} 
\def\RR{\mathbb{R}}
\def\ZZ{\mathbb{Z}} 
\def\SS{\mathbb{S}}
\def\TT{\mathbb{T}}

\def\CP{\CC\PP}

\DeclareMathOperator{\SL}{SL}

\DeclareMathOperator{\PGL}{PGL}
\DeclareMathOperator{\PU}{PU}

\def\D{\mathcal{D}}

\def\H{\mathcal{H}}

\def\N{\mathcal{N}}
\def\L{\mathcal{L}}
\def\M{\mathcal{M}}

\def\R{\mathcal{R}}
\def\T{\mathcal{T}}
\def\X{\mathcal{X}}

\def\Char{\mathfrak{X}}
\def\Charx{\Char^\times}

\def\ptorus{\TT_0}

\def\tri{\T}

\def\Ftess{F}
\def\Ftree{F^*}

\def\Hb{\overline{\H}}

\def\tetA{\mathfrak{T}_A}
\def\tetB{\mathfrak{T}_B}
\def\slab{\mathfrak{S}}

\def\cell{\D}

\def\sigmah{\hat{\sigma}}
\def\pih{\hat{\pi}}

\def\Cartan{\text{\AA}}

\DeclareMathOperator{\tr}{tr}

\DeclareMathOperator{\dev}{dev}
\DeclareMathOperator{\hol}{hol}

\DeclareMathOperator{\MCG}{MCG}

\def\bw{\mathbf{w}}

\def\n@te#1{\textsf{\boldmath \textbf{$\langle\!\langle$#1$\rangle\!\rangle$}}\leavevmode}
\def\note#1{\textcolor{red}{\n@te {#1}}}

\DeclareMathAlphabet{\mathcal}{OMS}{cmsy}{m}{n}

\theoremstyle{plain}

\newtheorem{thm}{Theorem}
\newtheorem*{thm*}{Theorem}
\newtheorem{lem}[thm]{Lemma}
\newtheorem*{lem*}{Lemma}

\newtheorem*{cor*}{Corollary}

\newtheorem*{cla*}{Claim}

\newtheorem*{pro*}{Proposition}
\newtheorem{rem}[thm]{Remark}
\newtheorem*{rem*}{Remark}

\newtheorem*{defn*}{Definition}

\newtheorem*{que*}{Question}

\addtocounter{MaxMatrixCols}{4}
\begin{document}
\title{Branched Cauchy-Riemann Structures on Once-Punctured\\ Torus Bundles}
\author{Alex Casella}

\begin{abstract}
Unlike in hyperbolic geometry, the monodromy ideal triangulation of a hyperbolic once-punctured torus bundle $M_f$ has no natural geometric realisation in Cauchy-Riemann (CR) space. By introducing a new type of $3$--cell, we construct a different cell decomposition $\cell_f$ of $M_f$ that is always realisable in CR space. As a consequence, we show that every hyperbolic once-punctured torus bundle admits a branched CR structure, whose branch locus is the set of edges of $\cell_f$. Furthermore, we explicitly compute the ramification order around each component of the branch locus and analyse the corresponding holonomy representations.
\end{abstract}
\primaryclass{57M50; 32V05}
\keywords{Geometric Structures, Cauchy, Riemann, Torus Bundles, Ideal Triangulations, Branching}

%%%%%%%%%%%%%%%%%%%%%%%%%%%%

\maketitle

%%%%%%%%%%%%%%%%%%%%%%%%%%%%
%%%%%%%%%%%%%%%%%%%%%%%%%%%%%

\section{Introduction}
\label{sec:introduction}
A \emph{geometry} or \emph{geometric structure} $(G,X)$ is a homogeneous space $X$ together with a transitive action on $X$ by a Lie group $G$, which acts as the symmetry group of the geometry. This concept was originally introduced by Klein in his celebrated Erlangen program~\cite{Klein-1872}, and rapidly developed by Ehresmann~\cite{Ehresmann-1936} and many others afterwards. When $X$ and $G$ are chosen appropriately, one recovers many classical geometries like hyperbolic $(\text{SO}(1,n),\HH^n)$, Euclidean $(\RR^n \rtimes \text{O}(n),\EE^n)$ or spherical $(\text{O}(n+1), \SS^n)$ geometry. A $(G,X)$--manifold $M$ is a manifold endowed with a $(G,X)$--structure, namely an atlas of charts in the model space $X$, whose transition functions are restrictions of elements of $G$.

As more and more connections between topology and geometry were discovered, $(G,X)$--structures have become a central topic in the study of manifolds. Among many contributors, William Thurston is one of the most celebrated pioneers. In~\cite{Thurston-1979}, he develops a way to construct hyperbolic structures on cusped $3$--manifolds using \emph{ideal triangulations}, namely decompositions into tetrahedra whose vertices are removed. The strategy consists in realising these simple pieces as hyperbolic objects, that glue up coherently in the manifold $M$. Consistency of the gluings can be encoded in a system of complex valued equations, whose solutions correspond to hyperbolic structures on $M$. Since Thurston, many authors have studied and further developed his technique~(\cite{CFT-2017}, \cite{Danciger-2013}, \cite{Fock-Goncharov-2006}, \cite{Neumann-Zagier-1985}, \cite{Segerman-Tillmann-2011}, \cite{Yoshida-1991}, et al.). 

In two recent papers~(\cite{Falbel-2008}, \cite{Falbel-Wang-2014}), a similar strategy was employed to construct branched Cauchy-Riemann structures (CR in short) on the complement of the figure eight knot. CR geometry is modelled on the three-sphere $\SS^3 \subset \CC^2$, with the contact structure obtained by the intersection $Y = T\SS^3 \cap JT\SS^3$, where $J$ is the multiplication by $i$ in $\CC^2$ (see for example~\cite{Bland-1994}). The operator $J$ restricted to $Y$ defines the standard CR structure on $\SS^3$. Its group of CR automorphisms is $\PU(2,1)$, thus a manifold $M$ has a (spherical) CR structure when it is endowed with a geometric $(\PU(2,1), \SS^3)$--structure. The fact that every $3$--manifold admits a contact structure~\cite{Martinet-1971} suggests that CR geometry has the potential to play an important role in three dimensional topology. Nevertheless, only few examples of CR manifolds are known. Most of them are closed Seifert fibred manifolds~\cite{Kamishima-Tsuboi-1991} or obtained by Dehn surgery from the Whitehead link~\cite{Schwartz-2003,Schwartz-2007}. On the other hand, some examples of $3$--manifolds which have no CR structures are known~\cite{Goldman-1983}. 

Inspired by the work of Falbel in~\cite{Falbel-2008}, we deal with a more general notion of CR structures, by allowing branching. Charts are not diffeomorphisms anymore, but locally branched coverings. By relaxing this condition, one obtains a geometric structure whose developing map is locally injective everywhere except for a nowhere-dense set, the branch locus. The spaces we investigate here are \emph{once-punctured torus bundles}, orientable manifolds which are the interior of compact $3$--manifolds with boundary a torus. They are fiber bundles over the circle, with fiber space a once-punctured torus. The figure eight knot complement is one such example. Most of these manifolds are hyperbolic~\cite{Otal-1996}, and exhibit important combinatorial properties. In particular, Floyd and Hatcher showed that each hyperbolic once-punctured torus bundle admits a canonical realisation as an ideal triangulation, called the \emph{monodromy ideal triangulation}~\cite{Floyd-Hatcher-1982}. This type of triangulation is part of a larger class of fundamental triangulations called \emph{veering triangulations}, developed by Agol in~\cite{Agol-2011}. The importance of this decomposition relies on its rich combinatorial structure, but also on its geometric properties. For example, Lackenby showed it to be geometrically canonical in the sense of Epstein-Penner~\cite{Lackenby-2003}, while Gu\'eritaud used it to recover Thurston's hyperbolicity of once-punctured torus bundles~\cite{Futer-Gueritaud-2006}.

In this paper we modify the monodromy ideal triangulation of each once-punctured torus bundle to a new ideal cell decomposition, that is geometrically realisable in CR space, and whose set of edges constitutes the branch locus. This decomposition is made up of tetrahedra and $3$--cells that we call \emph{slabs}, CW~complexes obtained by deformation retracting the base of a square pyramid onto one of its sides. In the case of the figure eight knot complement, Falbel~\cite{Falbel-2008} uses one of these slabs implicitly, as part of a \emph{generalised tetrahedron}, but the CR structure thus constructed consists of charts that are not embeddings of the tetrahedra. In particular, there is a small neighbourhood of one edge in a tetrahedron that develops to a flat bigon. This is not an obstruction in Falbel's proof: he focuses on the union of the images of two specific charts and shows that its quotient by the face pairings is homeomorphic to the figure eight knot complement. This strategy is hard to generalise to other punctured torus bundles and it is somehow unnatural. For example, it is true only for the figure eight knot complement that the branch locus occurs precisely at the edges of the triangulation. This suggests the use of a more suitable cell decomposition, such that we can geometrically realise each ideal cell by embedding it in CR space. For this to work, six geometrically different types of slabs will be defined. Each construction is very explicit and calculations are done directly in coordinates in the CR sphere. A collection of the main results is summarised in the following theorem.

\begin{thm}
\label{thm:intro}
	Let $M_f$ be a hyperbolic once-punctured torus bundle. Then $M_f$ admits an ideal cell decomposition $\cell_f$ that is geometrically realisable in CR space. It corresponds to a branched CR structure, whose branch locus is the set of edges of $\cell_f$.
	
	Moreover, the ramification order around each edge $e$ only depends on the valence of $e$ in $\cell_f$, and it is explicitly computable.
\end{thm}

The construction presented in this paper has the potential to further extend to more general punctured surface bundles, as they also admit layered triangulations. Although the realisability of the cell decomposition $\cell_f$ seems to rely on the fact that the base surface is a once-punctured torus, we intend to address this problem in future work using the veering triangulations of Agol~\cite{Agol-2011}.\\

The content of this paper is organised as follows. In sections~\ref{sec:once_punct_torus_bundles} and \ref{sec:monodromy_ideal_triangulation}  we review background material on once-punctured torus bundles and monodromy ideal triangulations. They mostly serve to set notations and underline the most relevant properties. CR geometry is covered in \S\ref{sec:CR_geometry}. There we define CR tetrahedra and slabs, the two fundamental $3$--cells which will be the building blocks of the CR structures in \S\ref{sec:CR_structure_bundles}. Section~\ref{sec:CR_structure_bundles} is the core of the paper, where we introduce the notion of branched CR structures and prove Theorem~\ref{thm:intro}, first in the explicit case of the figure eight knot, then in the general case for all once-punctured torus bundles. We conclude by computing ramification orders of the branch locus and a brief analysis of the holonomy representations in \S\ref{sec:property_structures}. In particular, the very last section~\ref{subsec:holonomy} is a summary of some facts about the holonomy representations and the connection to the work of Fock and Goncharov on positive representations~\cite{Fock-Goncharov-2006}, mostly for experts.

%%%%%%%%%%%%%%%%%%%%%%%%%%%%%
%%%%%%%%%%%%%%%%%%%%%%%%%%%%%

\section{Once-punctured torus bundles}
\label{sec:once_punct_torus_bundles}

Let $\ptorus := \left(\RR^2 \setminus \ZZ^2\right) / \ZZ^2$ be the once-punctured torus endowed with its standard differential structure and standard orientation. The \emph{mapping class group} of $\ptorus$ is the group $\MCG = \MCG(\ptorus)$ of isotopy classes of orientation preserving diffeomorphisms $f : \ptorus \rightarrow \ptorus$. For $[f] \in \MCG$, the \emph{once-punctured torus bundle} $M_f$ is the differentiable oriented $3$--manifold
$$
M_f := \ptorus \times [0,1] / \sim,
$$
where $(x,0) \sim (f(x),1)$ for $x \in \ptorus$. The manifold $M_f$ is a special fiber bundle over the circle, with fiber space $\ptorus$, well-defined up to diffeomorphism.

The natural identification of $\ptorus$ with the square spanned by the standard basis of $\RR^2$ induces an isomorphism $\MCG \cong \SL(2,\ZZ)$, hence each map $[f] \in \MCG$ has well-defined eigenvalues in $\CC$ (cf.~\cite{Farb-Margalit-2012}). This characterisation is fundamental to study the geometry of $M_f$, as for example it helps discerning hyperbolic bundles.

\begin{thm}[Thurston, 1996~\cite{Otal-1996}]
	\label{thm:hyperbolic_torus_bundles}
	$M_f$ admits a finite volume, complete hyperbolic metric if and only if $[f]$ has two distinct real eigenvalues.
\end{thm}

The element $[f]$ has distinct real eigenvalues if and only if $(\tr [f])^2 > 4$. If the trace is in $\{ -1,0,1\}$, then $[f]$ has finite order and $M_f$ is Seifert fibred. While if $\tr [f] = \pm 2$, then $f$ preserves a non-trivial simple closed curve in the punctured torus, which defines an incompressible torus or Klein bottle in $M_f$. In both cases we get an obstruction to the existence of the hyperbolic metric. An elementary and constructive proof of the other cases can be found in~\cite{Futer-Gueritaud-2006}.

%%%%%%%%%%%%%%%%%%%%%%%%%%%%%
%%%%%%%%%%%%%%%%%%%%%%%%%%%%%

\section{The monodromy ideal triangulation}
\label{sec:monodromy_ideal_triangulation}

In this section we recall the canonical realisation of a hyperbolic once-punctured torus bundle $M_f$ as an ideal triangulation, as described by Floyd and Hatcher in~\cite{Floyd-Hatcher-1982}, called the \emph{monodromy ideal triangulation} of $M_f$. For $M_f$ hyperbolic, Theorem~\ref{thm:hyperbolic_torus_bundles} implies that the eigenvalues of $[f]$ are distinct with the same sign. To simplify the construction, we are going to make the further assumption that the eigenvalues are positive. This will not cause any loss of generality: if $[f]$ has two negative eigenvalues, then $[-f]$ has positive eigenvalues, and the monodromy triangulation of $M_f$ can be easily deduced from the monodromy triangulation of $M_{-f}$. See Remark~\ref{rem:negative_eigenvalues_mon_tri} for more details.

%%%%%%%%%%%%%%%%%%%%%%%%%%%%%

\subsection{Flip sequence}
\label{subsec:flip_sequence}

An \emph{ideal triangulation} $\tri$ of $\ptorus$ is a maximal collection of pairwise disjoint and non-homotopic (relative the puncture) essential arcs. Every ideal triangulation of $\ptorus$ comprises three essential arcs, called \emph{ideal edges}, and divides the surface into two \emph{ideal triangles}. All of these ideal triangulations are combinatorially equivalent, but they can be distinguished by that they are not isotopic via an isotopy fixing the puncture.

Without loss of generality, one can assume that ideal triangulations of $\ptorus$ are \emph{straight}, in the sense that each ideal edge is the intersection with $\ptorus$ of the quotient of a straight line through the origin in $\RR^2$. In a straight triangulation $\tri$, the \emph{slope} of an edge is the slope of the corresponding straight line. Since edges start and terminate at the puncture, their slopes must be rational, hence there is a bijection between ideal edges and $\QQ \cup \{\infty\}$.

The set of isotopy classes of ideal triangulations can be encoded as the vertices of the \emph{Farey tree} $\Ftree$. This tree is dual to the \emph{Farey tessellation} $\Ftess$ (cf. Figure~\ref{fig:Farey_tessellation}), a tessellation of the hyperbolic plane by ideal triangles. The ideal vertices of this tessellation are the set of slopes of ideal edges $\QQ \cup \{\infty\}$ in the circle at infinity. In particular, the ideal vertices of a triangle in $\Ftess$ correspond to the slopes of three disjoint non-homotopic properly embedded arcs in $\ptorus$, and hence to an ideal triangulation. Thus, there is one vertex of the dual tree $\Ftree$ for each isotopy class of ideal triangulation of the once-punctured torus, and every such ideal triangulation is uniquely determined by a triplet of slopes satisfying the Farey sum. A beautiful treatment of this topic can be found in~\cite{Bonahon-2009}.

By adopting the convention that $0$ and $\infty$ are neither negative nor positive, we say that an ideal triangulation is \emph{positive} (resp. \emph{negative}) if at least one if its slopes is positive (resp. negative). The \emph{standard positive} (resp. \emph{negative}) \emph{ideal triangulation} of $\ptorus$ is the triangulation $\tri_+$ (resp. $\tri_-$) with slopes $\{0,1,\infty\}$ (resp. $\{0,-1,\infty\}$).

Two vertices of the dual tree $\Ftree$ are joined by an edge if and only if their corresponding ideal triangulations differ by a single slope. Passing from one triangulation to the other is usually called \emph{edge flipping}, as it involves removing one edge, resulting in a square with side identifications, and then inserting the other diagonal of the square. As $\Ftree$ is a tree, every two ideal triangulations of $\ptorus$ differ by a unique minimal sequence of edge flips.

Edge flips are of three types, depending on the slope we are flipping over. A \emph{right flip} $\R$ (resp. \emph{left flip} $\L$) is an edge flip of the largest (resp. smallest) slope. The remaining flip will be referred to as a \emph{middle flip $\M$}. For example, starting from the standard positive triangulation $\{0,1,\infty\}$ of $\ptorus$, a right flip produces the triangulation $\{0,\frac{1}{2},1\}$, a left flip gives $\{1,2,\infty\}$, and a middle flip gives $\{0,-1,\infty\}$.

One can visualise the dynamics of edge flips on the dual tree $\Ftree$ as follows. Let $\tri_{m}$ be a positive ideal triangulation (different from the standard one) and let $\tri_+,\tri_1,\dots,\tri_{m-1}$ be the sequence of triangulations along the unique shortest path between the standard positive triangulation and $\tri_m$. By definition, a middle flip kills the middle slope, hence it corresponds to a back-track towards $\tri_+$ and transforms $\tri_m$ into $\tri_{m-1}$, contradicting the minimality of the path. If you exclude back-tracking, one can move along $\Ftree$ in only two other ways, corresponding to a right or left flip. By orienting the hyperbolic plane with its standard positive orientation, a right (resp. left) flip corresponds exactly to turning right (resp. left) at $\tri_m$ (cf. Figure~\ref{fig:Farey_tessellation}). A perfectly analogous arguments works if we replace $\tri_{m}$ with a negative ideal triangulation.

\begin{figure}[ht]
	\centering
	\includegraphics[height=6.8cm]{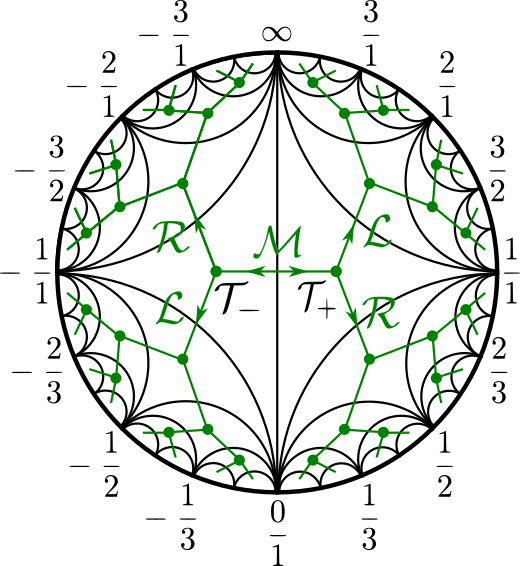}
	\caption{The Farey tree is dual to the Farey tessellation of the hyperbolic plane. Every vertex corresponds to an ideal triangulation of the once-punctured torus, and every edge corresponds to an edge flip.}
	\label{fig:Farey_tessellation}
\end{figure}

The following lemma is a direct consequence of the above discussion.

\begin{lem}
	\label{lem:solo_right_left_flips}
	Let $\tri$ be a positive (resp. negative) ideal triangulation different from the positive (resp. negative) standard one $\tri_0$. The unique sequence of edge flips from $\tri_0$ to $\tri_m$ does not contain any middle flips. Conversely, the sequence of flips from $\tri_m$ to $\tri_0$ only contains middle flips.
\end{lem}

Let $f : \ptorus \rightarrow \ptorus$ be a diffeomorphism of the once-punctured torus. The map $f$ acts transitively on the set of ideal triangulations of $\ptorus$, inducting an isomorphism of the Farey tree $\Ftree$. Every isomorphism of a simplicial tree has either a fixed point, or leaves invariant a unique copy of $\RR$, called \emph{axis}. The former case happens when $\tr([f]) \in \{-1,0,1\}$ and the action is periodic. In the latter case, let $V_0$ be a vertex on the axis. The unique shortest path in $\Ftree$ from $V_0$ to $f(V_0)$ runs along the axis, and naturally corresponds to a sequence of edge flips. When $\tr([f])^2 = 4$, the axis has a unique endpoint on the boundary of the hyperbolic plane, and the action is parabolic. Finally we observe that $-f$ acts on $\Ftree$ as $f$, hence we will only consider automorphisms with distinct positive real eigenvalues. 

After conjugating $f$, one can assume that $V_0$ corresponds to the standard positive ideal triangulation $\tri_0$ and the axis does not run through any negative triangulation. It follows from Lemma~\ref{lem:solo_right_left_flips} that $f(\tri_0)$ differs from $\tri_0$ by a unique sequence $\bw_f$ of right $\R$ and left $\L$ flips. Furthermore, when the eigenvalues of $f$ are distinct, $\bw_f$ always contains at least one right flip and one left flip. In other words, there exist $a_j,b_j,k \in \NN$ and $c \in \NN \cup \{0\}$ such that
$$
\bw_f = \R^{a_0} \L^{b_0} \dots \R^{a_k} \L^{b_k} \R^c \qquad \text{ or } \qquad \bw_f = \L^{a_0} \R^{b_0} \dots \L^{a_k} \R^{b_k} \L^c.
$$
We say that $\bw_f$ is the \emph{flip sequence} of $f$ or of $M_f$. Its \emph{length} is the total number of edge flips, namely $c + \sum_{j=0}^{k} (a_j+b_j)$. Under the canonical isomorphism $\MCG(\ptorus) \cong \SL_2(\ZZ)$, a right flip and a left flip correspond to the matrices
$$
[f_\R] = 
\begin{pmatrix}
1 & 1\\
0 & 1
\end{pmatrix} \qquad \text{ and } \qquad
[f_\L] = 
\begin{pmatrix}
1 & 0\\
1 & 1
\end{pmatrix}.
$$

%%%%%%%%%%%%%%%%%%%%%%%%%%%%%

\subsection{The triangulation}
\label{subsec:the_triangulation}

The following description of the monodromy ideal triangulation is adapted from~\cite{Futer-Gueritaud-2006}.\\

The \emph{standard ideal tetrahedron} $\sigma$ is, topologically, a compact tetrahedron with its vertices removed. One can picture $\sigma$ as a square with its two diagonals, as in Figure~\ref{fig:standard_tetrahedron}. Oriented simplices of $\sigma$ are determined by an ordering of the vertices, hence we refer to them by the notation $\sigma(i),\sigma(ij),\sigma(ijk),\sigma(ijkl)$. Sometimes we use the same notation for the unoriented counterparts, but only when it is clear from the context that we ignore the orientation. By identifying the pair of opposite edges $\sigma(13),\sigma(24)$ and $\sigma(12),\sigma(34)$, the exterior of $\sigma$ becomes the union of two \emph{pleated surfaces}, homeomorphic to the once-punctured torus $\ptorus$. The \emph{top} pleated surface $\sigma(\ptorus)_+$ is made up of the two ideal triangles $\sigma(143), \sigma(124)$, while the \emph{bottom} pleated surface $\sigma(\ptorus)_-$ is made up of the two ideal triangles $\sigma(123), \sigma(324)$. Thus the ideal triangulation of $\sigma(\ptorus)_+$ is obtained from $\sigma(\ptorus)_-$ by an edge flip along $\sigma(23)$.

Suppose $\ptorus$ is endowed with some ideal triangulation $\tri$. We say that the tetrahedron $\sigma$ \emph{layers on} $\ptorus$ if the bottom pleated surface of $\sigma$ is glued to $\ptorus$ via an orientation-preserving combinatorial isomorphism, called the \emph{layering}. Let $e$ be an oriented edge of $\tri$. We say that $\sigma$ layers on $\ptorus$ \emph{along} $e$ if the chosen layering identifies $e$ with the edge $\sigma(23)$. In general, there are six possible ways to layer $\sigma$ on $\ptorus$, one for each oriented edge of $\tri$. To simplify the notation we make a further distinction. We say that a layering of $\sigma$ is a (\emph{right}) $\R$ \emph{layering} (resp. (\emph{left}) $\L$ \emph{layering}) if $\sigma$ layers along the edge with largest (resp. smallest) slope, oriented towards (resp. away from) the origin in $\ptorus$. The motivation behind this notation is clear: if $\sigma$ right layers (resp. left layers) on $\ptorus$, the ideal triangulations of $\sigma(\ptorus)_+$ is obtained from $\sigma(\ptorus)_-$ by a right flip (resp. left flip).\\

{\noindent
	\begin{minipage}[t]{0.62\textwidth}
		\centering
		\includegraphics[height=3cm]{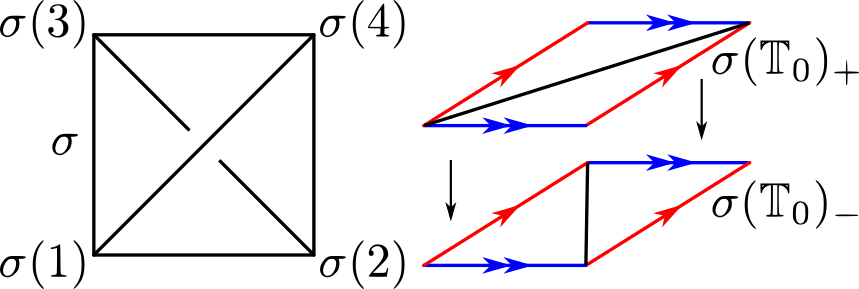}
		\captionof{figure}{The standard ideal tetrahedron $\sigma$ and the two pleated surfaces $\sigma(\ptorus)_+$ and $\sigma(\ptorus)_-$.}
		\label{fig:standard_tetrahedron}
	\end{minipage}\hspace{0.3cm}
	\begin{minipage}[t]{0.30\textwidth}
		\centering
		\includegraphics[height=3.3cm]{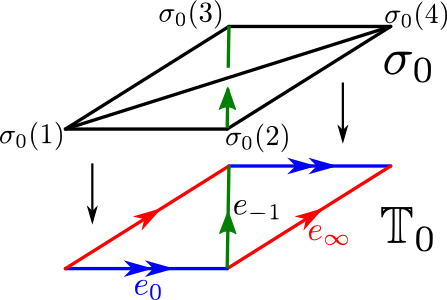}
		\captionof{figure}{A layering of the standard ideal tetrahedron along the edge $e_{-1}$.}
		\label{fig:tetrahedron_layering}
\end{minipage} }\vspace{0.5cm}

Let $f$ be an element of $\SL_2(\ZZ)$ with two distinct positive real eigenvalues and let $\bw_f$ be the flip sequence of $f$. Suppose $\bw_f$ has length $m$. Now we describe how to construct the monodromy triangulation of the hyperbolic once-punctured torus bundle $M_f$. Suppose $\ptorus$ is endowed with its negative standard ideal triangulation $\{0,-1,\infty\}$. Let $\sigma_0$ be a copy of the standard ideal tetrahedron layered on $\ptorus$ along the edge of slope $-1$, oriented as in Figure~\ref{fig:tetrahedron_layering}. Then the top pleated surface $\sigma_0(\ptorus)_+$ is triangulated as the positive standard ideal triangulation $\tri_0$. For each letter $\X_j$ in $\bw_f$, $j = 1,\dots,m$, reading from left to right, we perform an $\X_j$ layering of a copy of the standard ideal tetrahedron $\sigma_j$ on $\sigma_{j-1}(\ptorus)_+$. The space obtained by stacking these tetrahedra is naturally homeomorphic to $\ptorus \times I$. The last top pleated surface is $\sigma_{m}(\ptorus)_+$. Its triangulation $\tri_m$ is obtained from $\tri_0$ by performing the sequence of edge flips $\bw_f$. It follows that $\tri_m = f(\tri_0)$, and $f$ induces an identification between $\sigma_0$ and $\sigma_m$ which makes $\ptorus \times I$ into $M_f$. The \emph{monodromy triangulation} of $M_f$ is the ideal triangulation consisting of the tetrahedra $\sigma_0,\dots,\sigma_{m-1}$ and the face pairings inherited from the layering construction. As an example, see the monodromy ideal triangulation of the figure eight knot complement in \S\ref{subsec:figure_eight_complement}.

\begin{rem}
	\label{rem:negative_eigenvalues_mon_tri}
	We remark that $f$ and $-f$ act in the same way on the Farey tree, hence they share the same flip sequence. It follows that the monodromy triangulation of $M_{-f}$ differs from the one of $M_f$ only in the way $\sigma_0$ and $\sigma_m$ are identified. More precisely, one can construct $M_{-f}$ by composing the identification $f$ between $\sigma_0$ and $\sigma_m$ with a rotation by the angle $\pi$.
\end{rem}

The layering construction induces a natural cyclic ordering of the tetrahedra, thus they will often be indexed modulo $m$. Similarly, one should think of the flip sequence $\bw_f$ as a cyclic word, with a preferred starting point. For future reference, we introduce the following notation. A tetrahedron $\sigma_j$ of the monodromy triangulation is said to be of \emph{type} $\R$ (resp. \emph{type} $\L$) if the \emph{next} tetrahedron $\sigma_{j+1}$ is layered on top of it by a right (resp. left) layering. We will sometimes record the type of $\sigma_j$ by writing $\sigma_j^\R$ or $\sigma_j^\L$.

%%%%%%%%%%%%%%%%%%%%%%%%%%%%%

\subsection{Combinatorics around the edges}
\label{subsec:combinatorics_edges}

Let $\tri$ be the monodromy ideal triangulation of the once-punctured torus bundle $M_f$, and let $m$ be the length of its flip sequence $\bw_f$. Then $\tri$ is made up of $m$ tetrahedra $\sigma_0,\dots,\sigma_{m-1}$, glued together by the layering construction. We denote by $\pi$ the natural quotient map $\pi : \sqcup_j \sigma_j \rightarrow \tri \cong M_f$, defined by the face pairings. The space $M_f$ is the interior of a compact $3$--manifold with torus boundary, so its Euler characteristic is zero. It follows that $\tri$ has as many edges as tetrahedra, namely $m$. Nevertheless, each edge may be represented by multiple edges in each tetrahedron. The \emph{valence} of an edge is the size of its inverse image under $\pi$.

We are now going to describe the local structure of the edges in $\tri$. This will be useful in the analysis of the geometry around the edges in~\S\ref{sec:CR_structure_bundles}. We recall that each tetrahedron $\sigma_j$ is a copy of the standard ideal tetrahedron $\sigma$ via a canonical identification, hence it inherits labels at the vertices from $\sigma$.

Consider the edge $\sigma_0(14)$ of $\sigma_0$, and let $e_0 := \pi(\sigma_0(14) )$ in $\tri$. Suppose that $\sigma_0 = \sigma_0^\L$ is of type $\L$. Let $\sigma_1^\R,\dots,\sigma_{n_0}^\R$, $n_0 \geq 0$, be the (possibly empty) sequence of tetrahedra of type $\R$ layered on top of $\sigma_0^\L$, such that $\sigma_{n_0+1}^\L$ is of type $\L$. This sequence corresponds to a subsequence $\L \R^{n_0} \L$ in the word $\bw_f$ (thought of as a cyclic word). By definition, $\sigma_1$ left layers on $\sigma_0$, thus $\sigma_1(12),\sigma_1(34) \in \pi^{-1}(e_0)$. For every $2 \leq j \leq n_0+1$, the simplex $\sigma_j$ right layers on $\sigma_{j-1}$, therefore $\sigma_j(12),\sigma_j(34) \in \pi^{-1}(e_0)$. Finally, $\sigma_{n_0+2}$ left layers on $\sigma_{n_0+1}$, closing up the sequence of tetrahedra around $e_0$ with the edge $\sigma_{n_0+2}(23)$. Locally around $e_0$, the tetrahedra $\sigma_0,\dots,\sigma_{n_0+2}$ glue to form a \emph{ribbon}, where $\sigma_0$ and $\sigma_{n_0+2}$ appear once, while every other tetrahedron appears twice. See Figure~\ref{fig:edge_section_L} for a cross section of a neighbourhood of $e_0$. The simplex $\sigma_0$ (resp.~$\sigma_{n_0+2}$) is the \emph{bottom} (resp. \emph{top}) of the ribbon, and every other simplex $\sigma_j$ constitutes a \emph{loop} on each side. We deduce that the valence of $e_0$ is $2n_0 + 4$.

An analogous picture arises when we assume that $\sigma_0$ is of type $\R$, with the difference that every tetrahedron of type $\R$ is now of type $\L$, and vice versa (cf. Figure~\ref{fig:edge_section_R}). Furthermore, one may replace $\sigma_0$ with any other tetrahedron in $\tri$ and make the same definitions. For future reference, we summarise all of the above in the following Lemma.

\begin{lem}
	\label{lem:edge_ribbon}
	Every edge $e_j$ in $\tri$ corresponds to a unique subsequence $\L \R^{n_j} \L$ or $\R \L^{n_j} \R$ in $\bw_f$, $n_j \geq 0$, and a unique ribbon of tetrahedra $\sigma_j,\dots,\sigma_{j+n_j+2}$. The simplex $\sigma_j$ is the bottom of the ribbon, while $\sigma_{j+n_j+2}$ is the top of the ribbon, and every other tetrahedron in between constitutes a loop on each side. Hence the valence of $e_j$ is $2n_j + 4$.
\end{lem}

We remark that uniqueness of the ribbon follows from the fact that the bottom of the ribbon is the only tetrahedron in $\tri$ whose edge $(14)$ is a representative of $e_j$. Similarly, the top of the ribbon is the only tetrahedron whose edge $(23)$ belongs to $\pi^{-1}(e_j)$. A simple counting argument shows that there is a bijection between the set of  tetrahedra and the set of edges, thus associating every edge to its unique ribbon.\\

{ \noindent
	\begin{minipage}[t]{0.45\textwidth}
		\centering
		\includegraphics[width=4.5cm]{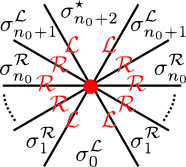}
		\captionof{figure}{A cross section of the ribbon around $e_0$ for $\sigma_0 = \sigma_0^\L$.}
		\label{fig:edge_section_L}
	\end{minipage} \hspace{1cm}
	\begin{minipage}[t]{0.45\textwidth}
		\centering
		\includegraphics[width=4.5cm]{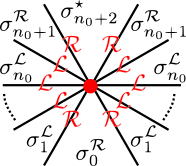}
		\captionof{figure}{A cross section of the ribbon around $e_0$ for $\sigma_0 = \sigma_0^\R$.}
		\label{fig:edge_section_R}
\end{minipage}}

%%%%%%%%%%%%%%%%%%%%%%%%%%%%
%%%%%%%%%%%%%%%%%%%%%%%%%%%%%

\section{CR Geometry}
\label{sec:CR_geometry}

The \emph{spherical Cauchy-Riemann geometry} is modelled on the \emph{CR sphere}, namely the three-sphere $\SS^3$ equipped with a natural $\mathrm{PU}(2,1)$ action. Unlike what we mentioned in the introduction, here we work with a definition of CR space that does not explicitly make use of contact geometry, but it underlines more clearly the action of $\mathrm{PU}(2,1)$. This point of view is going to be more suitable and relevant to our context. More details on the connection between CR geometry and contact geometry can be found in~\cite{Bland-1994}. For more background material and proofs of the following Lemmas we refer the reader to \S$4.3$~\cite{Goldman-1999} or \S$8$~\cite{Jacobowitz-1990}.

The matrix group $\text{U}(2,1)$ preserves the following Hermitian form defined on the complex space $\CC^3$:
$$
\langle z,w \rangle := \overline{w}^t  Jz, \qquad \text{where} \qquad 
J :=
\begin{pmatrix}
0   & 0  &    1 \\
0    & 1  &    0\\
1    & 0  &    0
\end{pmatrix}.
$$
Let $\pi :\CC^3\setminus\{ 0\} \rightarrow \CP^2$ be the
canonical projection, and consider the following cones in $\CC^3$,
$$
V_0 := \left\{ z\in { \CC}^{3} \setminus\{0\}\ \ |\ \
\langle z,z\rangle = 0 \ \right\}, \qquad V_-  := \left\{ z\in { \CC}^{3}\ \ |\ \ \langle z,z\rangle < 0
\ \right\}.
$$
Then $\HH_{\CC}^2 := \pi(V_-)$ is the Siegel domain model of the \emph{complex hyperbolic plane} and its boundary is
$$
\partial{\HH}_{ \CC}^{2} := \pi(V_0)=\{ [x,y,z]\in \CP^2\ |\ x\bar z+ |y|^2+z\bar x=0\ \}.
$$
As a topological space, $\partial{\HH}_{ \CC}^{2}$ is homeomorphic to the three-sphere $\SS^3.$ It is the \emph{spherical model} of the CR sphere. The projective group $\PU(2,1) := \text{U}(2,1) / \lambda I$ is the group of its biholomorphic transformations. The action of $\PU(2,1)$ on $\partial{\HH}_{ \CC}^{2}$ is by \emph{CR transformations}.

We are now going to describe a model for $\partial{\HH}_{ \CC}^{2}$ which is particularly suitable for our framework. The \emph{Heisenberg group} $\H$ is the space $\CC\times \RR$, equipped with the group law
$$
(z_1,t_1) \cdot (z_2,t_2) := (z_1 + z_2,\  t_1 + t_2 + 2 \Im(z_1 \overline{z_2})), \qquad z_1,z_2 \in \CC, \quad t_1,t_2 \in \RR.
$$
In the formula above, $\Im(z)$ is the imaginary part of the complex number $z$. Using stereographic projection $\Lambda$, one can identify $\partial{\HH}_{ \CC}^{2}$ with the one-point compactification $\Hb$ of $\H$, thus obtaining the \emph{Heisenberg model} of the CR sphere. In coordinates,
$$
\Lambda : \begin{bmatrix}
x \\ y \\ 1
\end{bmatrix} \mapsto\left( y, \frac{2x + |y|^2}{i} \right), \quad
\Lambda^{-1} : (z,t) \mapsto \begin{bmatrix}
\frac{i t - |z|^2}{2} \\ z \\ 1
\end{bmatrix}
\qquad \text{ and } \qquad \Lambda :  \begin{bmatrix}
1 \\ 0 \\ 0
\end{bmatrix} \mapsto \infty.
$$
The action of $\PU(2,1)$ on $\Hb$ is by defined by conjugating with $\Lambda$.

Complex geodesics in $\HH_{ \CC}^{2}$ are totally geodesic submanifolds of real dimension two. Their boundaries in $\partial{\HH}_{ \CC}^{2}$ are topological circles, called \emph{$\CC$--circles}. A $\CC$--circles in $\Hb$ is the image under $\Lambda$ of a $\CC$--circles in $\partial{\HH}_{ \CC}^{2}$.

\begin{lem}
	\label{lem:C_circles_Heisenberg_space}
	In the Heisenberg model $\overline{\H}$, $\CC$--circles are either vertical lines or ellipses whose projections onto the $z$--plane are circles.
\end{lem}

We remark that a complex geodesic in $\HH_{ \CC}^{2}$ is naturally endowed with a positive orientation given by its complex structure, hence every $\CC$--circle also inherits an orientation.

\begin{lem}
	\label{lem:CR_transformations_preserve_circles}
	CR transformations map $\CC$--circles to $\CC$--circles, preserving their orientations.
\end{lem}

Given two distinct points in Heisenberg space $\Hb$, there is a unique $\CC$--circle between them. We say that $m$ points of $\Hb$ are in \emph{general position} if no three are contained in the same $\CC$--circle. The group of CR transformations acts transitively on pairs of distinct points, while generic configurations of triples of points are parametrised by a real number. Given a cyclically ordered triple of points $\cyclic{P_1,P_2,P_3}$ in $\Hb$, its \emph{Cartan angle} $\Cartan$ is
$$
\Cartan(P_1,P_2,P_3) := \arg( - \langle P'_1, P'_2 \rangle \langle P'_2, P'_3 \rangle \langle P'_3, P'_1 \rangle  ) \in \RR, \qquad \text{ where } \quad P'_j = \Lambda^{-1}(P_j).
$$

\begin{lem}
	\label{lem:cartan_angle}
	Three points in $\Hb$ are not in general position if and only if their Cartan angle is $0$. Moreover, the group $\PU(2,1)$ is simply transitive on ordered triples of points in general position with the same Cartan angle.
\end{lem}

%%%%%%%%%%%%%%%%%%%%%%%%%%%%%

\subsection{CR Edges}
\label{subsec:CR_edges}

Given two distinct points $P_1,P_2 \in \Hb$, the \emph{oriented edge} $[P_1,P_2]$ is the segment of the $\CC$--circle between $P_1$ and $P_2$, oriented towards $P_2$. For example, the oriented edge $[(0,0),\infty]$ is the segment $\{ (0,t) \in \overline{\H} \ | \ t > 0 \}$, oriented towards $\infty$. Then $[P_1,P_2] \cup [P_2,P_1]$ is the whole $\CC$--circle through $P_1$ and $P_2$. A disk bounded by the loop $[P_1,P_2] \cup [P_2,P_1]$ will be referred to as a \emph{bigon}.

%%%%%%%%%%%%%%%%%%%%%%%%%%%%%

\subsection{CR Triangles}
\label{subsec:CR_triangles}

Suppose $P_1,P_2,P_3 \in \overline{\H}$ are three points in general position. For each pair, there are two possible oriented edges, for a total of eight choices of $1$--skeletons defining a \emph{triangle}. As $\overline{\H}$ is simply connected, we can always extend the $1$--skeleton of a triangle to an embedded $2$--cell, with boundary defined by that $1$--skeleton. This can be done in many different ways, all equivalent up to isotopy. Inspired by the work of Falbel~\cite{Falbel-2008}, we define the \emph{marked triangles} $[{P_1}^+,P_2,P_3]$ and $[{P_1}^-,P_2,P_3]$ as foliations of oriented edges  (cf.~Figure~\ref{fig:CR_triangles}):
\begin{align*}
[{P_1}^+,P_2,P_3] &:= \{ P \in \overline{\H} \ | \ P \in [P_1,P_t] \ \text{ for } \ P_t \in [P_2,P_3]  \},\\
[{P_1}^-,P_2,P_3] &:= \{ P \in \overline{\H} \ | \ P \in [P_t,P_1] \ \text{ for } \ P_t \in [P_2,P_3]  \}.
\end{align*}

\begin{figure}[!h]
	\centering
	\includegraphics*[height=4cm]{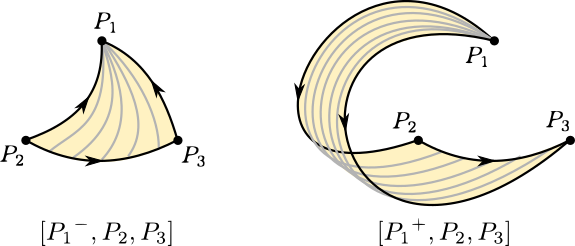}
	\caption{Marked triangles are foliated by oriented edge.}
	\label{fig:CR_triangles}
\end{figure}

By fixing $P_1$ to be at infinity, a marked triangle is half a cylinder with base part of a finite $\CC$--circle. One of the advantages of using marked triangles is that they are uniquely determined by their vertices. The following result is a direct consequence of Lemma~\ref{lem:CR_transformations_preserve_circles}.

\begin{lem}
	\label{lem:CR_transformations_triangles}
	Let $P_1,P_2,P_3$ and $Q_1,Q_2,Q_3$ be two triples of points of $\overline{\H}$ in general position. Suppose there exists $G \in \PU(2,1)$ such that $G(P_j) = Q_j$, for all $j \in \{1,2,3\}$. Then
	$$
	G([{P_1}^\star,P_2,P_3]) = [{Q_1}^\star,Q_2,Q_3], \qquad \star \in \{+,-\}.
	$$
\end{lem}

%%%%%%%%%%%%%%%%%%%%%%%%%%%%%

\subsection{CR Tetrahedra and Slabs}
\label{subsec:CR_tetrahedra_slabs}

Given four points of $\Hb$ in general position, a choice of a marked triangle for each triple will not always patch up to form the boundary of a $3$--simplex. On one hand, the faces might not be compatible at the edges and have gaps between them. On the other hand, they could intersect away from the edges. One quickly finds that there is not a canonical choice of marked triangles which always works, thus three dimensional simplices need to be checked on a case by case basis.

Here we are going to describe two fundamental $3$--cells, which will be the building blocks of the CR structures in~\S\ref{sec:CR_structure_bundles}. They are subsets of the Heisenberg space, both topologically homeomorphic to the $3$--ball, but equipped with different simplicial structures. These spaces are defined to be especially symmetric, in the sense that several of their faces can be glued pairwise with monotone maps (cf. Lemma~\ref{lem:face_pairings_are_monotone}). That is not always the case for generic triangles, as previously underlined in Lemma~\ref{lem:cartan_angle}.

\textbf{The standard symmetric tetrahedron.} Let $\omega$ be the cube root of unity $\omega = -\frac{1}{2}\left( 1 + i \sqrt{3}\right)$. We consider the following $4$--tuple of points in general position in Heisenberg space:
$$
P_1 := (1,\sqrt{3}), \qquad P_2 := (-\omega,\sqrt{3} ), \qquad P_3  := (0,0), \qquad P_4 := \infty.
$$

For each triple of points, we consider the following marked triangles:
\begin{enumerate}
	\item $[{P_4}^-,P_1,P_2]$: the oriented segment $[P_1,P_2]$ is the shortest arc of the circle $(e^{i \theta},\sqrt{3})$, oriented from $P_1$ to $P_2$. The triangle $[{P_4}^-,P_1,P_2]$ is part of a cylinder, foliated by vertical segments above $[P_1,P_2]$.
	\item $[{P_4}^-,P_3,P_1]$: the edge $[P_3,P_1]$ is an arc of ellipse which projects onto the $z$--coordinate of the Heisenberg space as an arc of the unit circle with centre $-\omega$. It is given by the parametrisation
	$$
	[P_3,P_1] := \left( -\omega + e^{is}, \sqrt{3} \cos(s) - \sin(s) \right), \qquad s : -\frac{2\pi}{3} \mapsto -\frac{\pi}{3}.
	$$
	Hence $[{P_4}^-,P_3,P_1]$ is foliated by the vertical rays from $[P_3,P_1]$ to $P_4$.
	\item $[{P_4}^-,P_3,P_2]$: this marked triangle is obtained by a $\frac{\pi}{3}$ clockwise rotation of the previous triangle $[{P_4}^-,P_3,P_1]$.
	\item $[{P_2}^-,P_3,P_1]$ and $[{P_3}^+,P_1,P_2]$: the first marked triangle is foliated by oriented edges from $[P_3,P_1]$ to $P_2$. For $\varphi(t,s) := t+s+\frac{\pi}{3}$, we have
	{\fontsize{11}{11}\selectfont 
		\begin{align*}
		[{P_2}^-&,P_3,P_1] := \left( e^{i \varphi(t,s)} + e^{i(s-\frac{\pi}{3})} -\omega  \right. ,\\
		& \left. -\sin(\varphi(t,s)) - \sin(\varphi(t,0)) + \sin(s) + \sqrt{3}\left( \cos(\varphi(t,s)) - \cos(t,0) + \cos(s) + 1 \right) \right),
		\end{align*}
	}
	where $s : -\frac{2\pi}{3} \mapsto -\frac{\pi}{3}$ and $t : 0 \mapsto \frac{\pi}{3}$. The latter one instead, is foliated by oriented edges from $P_3$ to $[P_1,P_2]$. It can be parametrised as
	$$
	[{P_3}^+,P_1,P_2] := \left( e^{it} \left(-\omega + e^{is}\right), \sqrt{3} \cos(s) - \sin(s) \right), \quad s : -\frac{2\pi}{3} \mapsto -\frac{\pi}{3}, \quad t : 0 \mapsto \frac{\pi}{3}.
	$$
\end{enumerate}

\begin{lem}(\cite{Falbel-2008})
	\label{lem:faces_of_tetrahedra}
	The spaces
	\begin{align}
	&[{P_4}^-,P_1,P_2] \cup [{P_4}^-,P_3,P_1] \cup [{P_4}^-,P_3,P_2] \cup [{P_2}^-,P_3,P_1],  \label{eq:tetrahedron_type_A}\\
	&[{P_4}^-,P_1,P_2] \cup [{P_4}^-,P_3,P_1] \cup [{P_4}^-,P_3,P_2] \cup [{P_3}^+,P_1,P_2],  \label{eq:tetrahedron_type_B}
	\end{align}
	are combinatorially isomorphic to a $3$--simplex. In particular, they bound a $3$--ball on each side in $\overline{\H}$.
\end{lem}
The \emph{standard (symmetric) tetrahedron $\tetA$ of type $A$} is the closure of the $3$--ball bounded by the $3$--simplex in~(\ref{eq:tetrahedron_type_A}), which is contained in the upper half of $\overline{\H}$. Similarly, the $3$--simplex in~(\ref{eq:tetrahedron_type_B}) is the boundary of the \emph{standard (symmetric) tetrahedron $\tetB$ of type $B$}. Figure~\ref{fig:CR_tetrahedra} shows $\tetA$ and $\tetB$ in the Heisenberg model.

\begin{figure}[ht]
	\centering
	\includegraphics[height=6cm]{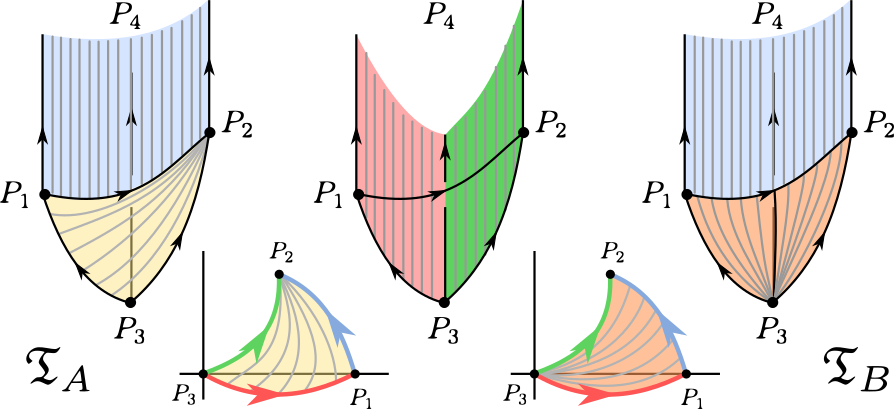}
	\caption{The standard symmetric tetrahedra $\tetA$ and $\tetB$ only differ along the face with vertices $\left\{ P_1, P_2, P_3 \right\}$. Their standard embeddings in Heisenberg space and their $\CC$--projections are displayed here.}
	\label{fig:CR_tetrahedra}
\end{figure}

These tetrahedra exhibit various symmetries, for example an anti-holomorphic involution swapping the vertices $P_1$ with $P_2$, and $P_3$ with $P_4$ (cf.~\cite{Will-2006}). Furthermore, the vertices of each face (taken with the correct cyclic order) have the same Cartan angle,
$$
\Cartan(P_2,P_3,P_1) = \Cartan(P_4,P_1,P_2) = \Cartan(P_4,P_3,P_2) = \Cartan(P_4,P_3,P_1) = \frac{\pi}{3}.
$$
As a consequence of Lemma~\ref{lem:cartan_angle} and Lemma~\ref{lem:CR_transformations_triangles}, we can glue faces of $\tetA$ and $\tetB$ pairwise by (unique) CR transformations. Consider the following matrices of $\PU(2,1)$,
$$
G_1 := \begin{bmatrix}
-\omega & 0 & 0\\
1 & 1 & 0\\
-\overline{\omega} & \omega & -\omega
\end{bmatrix}, \qquad
G_2 := \begin{bmatrix}
1 & 1 & \omega\\
0 & -\overline{\omega} & \overline{\omega}\\
0 & 0 & 1
\end{bmatrix}, \qquad
G_3 := \begin{bmatrix}
1 & 0 & 0\\
0 & -\omega& 0\\
0 & 0 & 1
\end{bmatrix}.
$$
These are the unique CR transformations mapping:
\begin{alignat*}{3}
G_1 : \quad &P_4 \mapsto P_2 \quad P_3 \mapsto P_3 \quad P_1 \mapsto P_1 \quad \text{ hence } \quad &&[{P_4}^-,P_3,P_1] \mapsto [{P_2}^-,P_3,P_1],\\
G_2 : \quad &P_4 \mapsto P_4 \quad P_1 \mapsto P_3 \quad P_2 \mapsto P_2 \quad \text{ hence } \quad &&[{P_4}^-,P_1,P_2] \mapsto [{P_4}^-,P_3,P_2],\\
G_3 : \quad &P_4 \mapsto P_4 \quad P_3 \mapsto P_3 \quad P_1 \mapsto P_2 \quad \text{ hence } \quad &&[{P_4}^-,P_3,P_1] \mapsto [{P_4}^-,P_3,P_2].
\end{alignat*}
We remark that $G_2$ and $G_3$ are face pairings between two standard tetrahedra of any types, while $G_1$ necessarily glues onto a face of the standard tetrahedron of type $A$. Furthermore, $G_2$ and $G_3$ can be described quite nicely in Heisenberg coordinates:
\begin{align*}
G_2([z,t]) &= \left[-\overline{\omega}(z-1) \ , \ \sqrt{3}\overline{\omega}(z+\overline{\omega}-1)(z+\omega) + t \right],\\
G_3([z,t]) &= \left[-\omega z \ , \  t \right].
\end{align*}
The transformation $G_2$ preserves vertical $\CC$--circles and it restricts on the $z$--plane to a $\frac{\pi}{3}$ clockwise rotation around the point $-\omega$. The transformation $G_3$ is a $\frac{\pi}{3}$ anticlockwise rotation of $\overline{\H}$ around the vertical $\CC$--circle through $[0,0]$.

\textbf{The slabs.} The next fundamental piece that we are going to define is of the combinatorial type of the CW~complex obtained by deformation retracting the base of a square pyramid onto one of its sides. In particular, it is a $3$--cell bounded by two triangular faces and two bigons. It contains a total of five $1$--cells and three $0$--cells.

We define the following bigons of $\overline{\H}$:
\begin{alignat*}{4}
&B' := &&\left( 1 + t e^{-i\frac{\pi}{6}}, s \right), \qquad &&t \in \RR_{>0} \cup \{\infty\}, \quad s \in \RR \cup \{\infty\},\\
&B_k := &&\left( -\omega + t e^{i\frac{\pi}{6}(1-2k)}, s \right), \qquad &&t \in \RR_{>0} \cup \{\infty\}, \quad s \in \RR \cup \{\infty\}, \quad k \in \ZZ.
\end{alignat*}
We remark that both $B'$ and $B_k$ are foliated by vertical $\CC$--circles. In particular, $B' \cap B_k = \infty$ for all $k$. Moreover,
$$
B_{k_1} = B_{k_2} \iff k_1 = k_2 \mod 6.
$$
The CW~complex obtained by attaching
$$
[{P_4}^+,P_1,P_2] \ \cup \ [{P_4}^-,P_1,P_2]  \ \cup \  B'  \ \cup \  B_k,
$$
is topologically a $2$--sphere. For all $k$, it bounds a $3$--ball containing the point $(2,\sqrt{3}) \in \overline{\H}$. We define the \emph{slab} $\slab_k$ to be the closure of such $3$--ball. The slabs $\slab_{k_1}$ and $\slab_{k_2}$ are geometrically equivalent if and only if $k_1 = k_2 \mod 6$, in the sense that there is $G \in \PU(2,1)$ such that $G (\slab_{k_1}) = \slab_{k_2}$. This is due to the fact that the $2$--skeletons of $\slab_{k_1}$ and $\slab_{k_2}$ only differ along one face. Whence we defined a total of six different slabs. Two examples $\slab_1$ and $\slab_4$ are depicted in Figure~\ref{fig:slabs}.

\begin{figure}[!t]
	\centering
	\includegraphics[height=5cm]{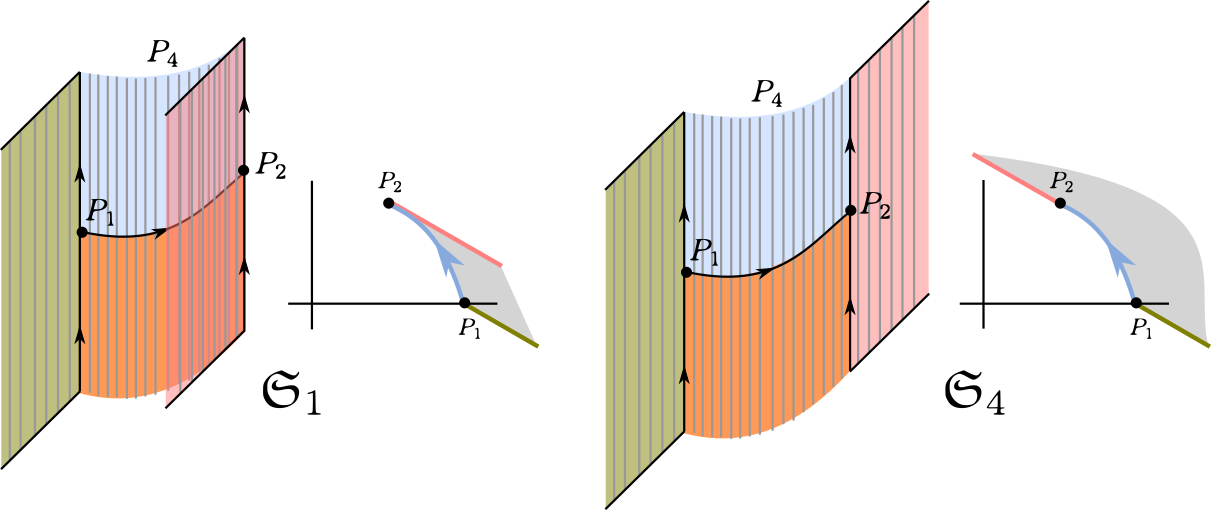}
	\caption{Standard embeddings in Heisenberg space and $\CC$--projections of the slabs $\slab_1$ and $\slab_4$.}
	\label{fig:slabs}
\end{figure}

As we mentioned earlier, $\Cartan(P_4,P_1,P_2) = \Cartan(P_3,P_1,P_2)$, hence let $G_4$ be the (unique) element of $\PU(2,1)$: \vspace*{0.4cm}

\begin{minipage}[0.0\textwidth]{0.4\textwidth}
	$$
	G_4 := \begin{bmatrix}
	0 & 0 & -\omega\\
	0 & -\overline{\omega} & 0\\
	-\omega & 0 & 1-\overline{\omega}
	\end{bmatrix},
	$$
\end{minipage}
\begin{minipage}[0.5\textwidth]{0.3\textwidth}
	\vspace*{-0.5cm}
	\begin{alignat*}{1}
	G_4 : P_4 \mapsto P_3 \quad P_1 &\mapsto P_1 \quad P_2 \mapsto P_2,\\
	[{P_4}^+,P_1,P_2] &\mapsto [{P_3}^+,P_1,P_2].
	\end{alignat*}
\end{minipage} \vspace*{0.5cm}

For all $k$, the CR transformation $G_4$ is a face pairing between the slab $\slab_k$ and the standard tetrahedron of type $B$.

The use of six different slabs turns out to be necessary in the general construction of~\S\ref{subsec:general_case}. The reason for the number six is due to the fact that the CR transformations $G_1, G_2, G_3$ and $G_4$ are all of order six. The connection between them and the slabs is revealed in Theorem~\ref{thm:geometric_realisation_is_structure}.\\

We conclude this section with a definition and an observation. Let $W_1$ and $W_2$ be two CW~complexes embedded in $\overline{\H}$, and let $G \in \PU(2,1)$ be a face pairing between the faces $F_1 \subset W_1$ and $F_2 \subset W_2$. Then $G(W_1)$ and $W_2$ might intersect away from $G(F_1) = F_2$. We say that the face pairing $G$ is \emph{monotone} if there are neighbourhoods $\N_1,\N_2$ of $F_1,F_2$ in $W_1,W_2$ respectively such that $\N_2 \cap G(W_1) = G(\N_1) \cap W_2   = F_2$. The following result generalises an observation by Falbel~\cite{Falbel-2008}.

\begin{lem}
	\label{lem:face_pairings_are_monotone}
	The transformations $G_1,G_2,G_3$ are monotone face pairings of the standard symmetric tetrahedra $\tetA$ and $\tetB$, while $G_4$ is a monotone face pairing between the slab and the standard tetrahedron of type $B$.
\end{lem}

\begin{proof}
	The transformations $G_2$ and $G_3$ are simple to check. They preserve vertical $\CC$--circles, therefore one only needs to check the intersection of the projections of the tetrahedra on the $z$--plane.
	
	On the other hand, $G_1$ and $G_4$ are more tedious. We give a summary of the argument for $G_4$, and refer to~\cite{Falbel-2008} for $G_1$. Consider the slab $\slab_k$ and the tetrahedron $\tetB$. The transformation $G_4^{-1}$ glues $\tetB$ to $\slab_k$ along the face $[{P_4}^+,P_1,P_2] = G_4^{-1}([{P_3}^+,P_1,P_2])$. The remaining vertex of $\tetB$ is mapped to the point $G_4^{-1}(P_4) = [0,2\sqrt{3}]$ in Heisenberg space. The projection of the $1$--skeleton of $G_4^{-1}(\tetB)$ is displayed next to the projection of $\slab_1$ in Figure~\ref{fig:projection_monotonicity}.
	
	\begin{figure}[!h]
		\centering
		\includegraphics[height=3.5cm]{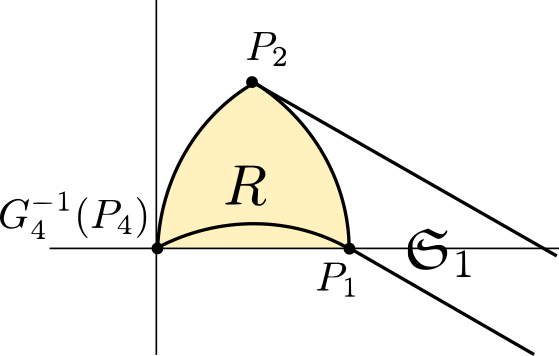}
		\caption{The projection of the $1$--skeleton of $G_4^{-1}(\tetB)$ next to the projection of $\slab_1$.}
		\label{fig:projection_monotonicity}
	\end{figure}
	
	Let $R$ be the region of $\CC$--plane bounded by the straight segment from $0$ to $1$, and the projections of the edges $[P_1,P_2]$ and $G_4^{-1}([P_4,P_2])$. Then $G_4^{-1}(\tetB)$ is completely contained in the vertical cylinder of Heisenberg space with base $R$. In particular, there is a neighbourhood of the common face where $G_4^{-1}(\tetB)$ and $\slab_k$ only intersect along the face, and therefore $G_4^{-1}$ is a monotone face pairing between $\slab_k$ and $\tetB$. By symmetry of the definition, we conclude that $G_4$ is also monotone.
\end{proof}

%%%%%%%%%%%%%%%%%%%%%%%%%%%%%
%%%%%%%%%%%%%%%%%%%%%%%%%%%%%

\section{Branched CR structures on once-punctured torus bundles}
\label{sec:CR_structure_bundles}

Let $M_f$ be a hyperbolic once-punctured torus bundle. In this section we prove the main result of this paper, that $M_f$ admits a branched CR structure (cf. Theorem~\ref{thm:geometric_realisation_is_structure}). We start by formalising the notion of a branched CR structure on $M_f$. Definitions and terminology are inspired by the work on branched analytic structures on Riemann surfaces in~\cite{Mandelbaum-1972}. Then we describe CR structures as finite geometric realisations of ideal decompositions. Finally, we give the construction for the figure eight knot~\S\ref{subsec:figure_eight_complement} and in the general case~\S\ref{subsec:general_case}.\\

A \emph{branched covering} between two manifolds is a covering map everywhere except for a nowhere-dense set, called the \emph{branch locus}. For example, the \emph{standard CR branching map} $\xi : \overline{\H} \rightarrow \overline{\H}$ defined by $\xi(z,t) := (z^N,t)$ is a branched map of ramification order $N \in \ZZ \setminus \{0\}$. In particular, $\xi$ is locally injective everywhere except at the branch locus, namely the Heisenberg $t$--axis, where the total angle is $2N\pi$.

A \emph{CR branched coordinate covering} $\{ U_j, \phi_j  \}$ of $M_f$ consists of an open covering $\{ U_j \}$ of $M_f$ together with branched coverings $\phi_j : U_j \rightarrow V_j$ into open subsets $V_j$ of the CR space $\overline{\H}$, that are locally modelled on the standard CR branching map $\xi$. A \emph{branched CR cover} is a coordinate covering $\{ U_j, \phi_j  \}$ such that, on each non-empty intersection $U_i \cap U_j$, there are homeomorphisms called \emph{coordinate transition functions}
$$
G_{ij} : \phi_i(U_i \cap U_j) \rightarrow \phi_j(U_i \cap U_j),
$$
that are restrictions of elements in $\PU(2,1)$. In particular they satisfy $G_{ij} \circ \phi_i = \phi_j$. A \emph{branched CR structure} on $M_f$ is an equivalence class of branched CR covers, where two branched CR covers are equivalent if their union is a branched CR cover. As a brief example of a natural branched structure, we mention the hypersurface $\Sigma \subset \CC^2$ defined by
$$
\Sigma := \{ (z_1,z_2) \in \CC^2 \ | \ |z_1|^{2N} + |z_2|^2 = 1 \}.
$$
We observe that the map $\xi' : \Sigma \rightarrow \overline{\H}$ defined by $\xi'(z_1,z_2) = (z_1^N,z_2)$ is a branched covering, branched along the curve $z_2=0$.

Let $\{ U_j, \phi_j  \}$ be a branched CR structure on $M_f$. When the ramification order of each chart $\phi_j$ is one, they are homeomorphisms and one recovers the usual definitions of coordinate covering, CR cover and CR structure~\cite{Thurston-1979}. We recall that every CR structure admits a \emph{developing map} and a \emph{holonomy representation},
$$
\dev : \widetilde{M_f} \rightarrow \overline{\H} \qquad \text{ and } \qquad \hol : \pi_1(M_f) \rightarrow \PU(2,1),
$$
such that
\begin{equation}
\label{eq:dev_hol_equivariancy}
\hol(\gamma) \cdot \dev(x) = \dev( \gamma \cdot x), \qquad \gamma \in \pi_1(M_f), \quad x \in \widetilde{M_f}.
\end{equation}
The developing map is considered up to deck transformation invariant isotopy, and the pair $(\dev,\hol)$ is uniquely determined up to the following action of $\PU(2,1)$:
$$
G \cdot (\dev,\hol) := (G \cdot \dev, \ G \cdot \hol \cdot G^{-1}), \qquad G \in \PU(2,1).
$$
Developing maps thus obtained are locally injective, as the charts $\phi_j$ are homeomorphisms. Vice versa, a locally injective developing map together with a holonomy representation satisfying the equivariancy condition (\ref{eq:dev_hol_equivariancy}), always defines a CR structure. We refer the reader to~\cite{Thurston-1997} for a full treatment in the wider context of geometric $(G,X)$--structures.

In a similar fashion, one may construct developing maps and holonomy representations for branched CR structures. From the motivational point of view, given only a representation into $\PU(2,1)$, it is not clear that it occurs as the holonomy representation of a spherical CR structure. In that sense, it is useful to consider the more general definition of a branched structure, in the hope that any given representation might be understood in a geometric way. The only difference being that developing maps are not locally injective but locally branched coverings. In particular, the holonomy around each connected component of the branch locus is a rotation by an integer multiple of $2\pi$, and therefore trivial, ensuring a well defined representation of $\pi_1(M_f)$.\\

%%%%%%%%%%%%%%%%%%%%%%%%%%%%%

\subsection{Finite geometric realisations}
\label{subsec:finite_geometric_realisations}

In~\S\ref{subsec:general_case} we construct special branched CR structures on $M_f$, whose branch locus is a disjoint union of curves. The strategy is to use an \emph{ideal cell decomposition} $\cell_f$ of $M_f$, modelled on its monodromy ideal triangulation $\tri_f$, whose edge set is the branch locus. We are going to realise each ideal cell as a geometric object in Heisenberg space and each face pairing as an element of $\PU(2,1)$, in a compatible fashion. More precisely, suppose $\cell_f$ is made up of the ideal $3$--cells $\sigma_i$, with face pairings $g_j$. We recall that a face pairing is called monotone when the paired cells only intersect along the common face in a neighbourhood of such face (cf. end of~\S\ref{subsec:CR_tetrahedra_slabs}). A \emph{geometric realisation} $\{ \phi_i, G_j \}$ of $\cell_f$ in $\overline{\H}$ consists of embeddings $\phi_i: \sigma_i \rightarrow \overline{\H}$ and CR transformations $G_j \in \PU(2,1)$, satisfying the following condition: if $g_j$ is the gluing map between the faces $F_i$ and $F_k$ of the ideal $3$--cells $\sigma_i$ and $\sigma_k$ respectively, then $G_j$ is a monotone CR transformation pairing $\phi_i(F_i)$ and $\phi_k(F_k)$ in the same combinatorial way. Then we say that $\phi_i$ and $G_j$ are \emph{geometric realisations} of $\sigma_i$ and $g_j$ respectively.

A geometric realisation differs from a branched CR structure only at the edges. For each edge $e$, consider a small oriented loop $\gamma_e$ around $e$, with prescribed starting point $x \in \gamma_e$ contained in the interior of some cell. Let $F_0^e \dots F_{N_e}^e$ be the sequence of faces in $\cell_f$ containing $e$, ordered as they are crossed by $\gamma_e$, starting from $x$. As $\gamma_e$ travels through a face $F^e_j$, it leaves an ideal cell $\sigma$ to enter another ideal cell $\sigma'$ (possibly equal to $\sigma$). Let $g_j^e$ be the face pairing gluing $\sigma$ to $\sigma'$ along $F^e_j$, and let $G_j^e$ be its corresponding geometric realisation. Then the \emph{geometric holonomy} of $\{ \phi_i, G_j \}$ along $\gamma_e$ is the product $\prod_{j=0}^{N_e} G_{N_e - j}^e$. We remark that a different choice of $\gamma_e$ only changes the geometric holonomy by conjugation or by inverse, hence whether the geometric holonomy around an edge $e$ is trivial (namely equal to the identity) or not, does not depend on the choice of $\gamma_e$.

In general, it is not guaranteed that the geometric holonomy is trivial because a geometric realisation does not enforce any conditions on the local structure around the edges. However, when that is the case for every edge of the cell decomposition, then a geometric realisation can be extended to a branched CR structure. More precisely, there is a branched CR structure on $M_f$ whose set of charts include the embeddings $\phi_i$, and the coordinate transition functions along the faces are the CR transformations $G_j$. In particular, it is important that the maps $G_j$ are monotone to ensure local injectivity at the faces. Furthermore, the fact that the geometric holonomy around an edge $e$ is trivial allows the construction of a chart containing $e$ which is a branched covering (with branch locus $e$) and which agrees with $\phi_i$ around $e$. An example of this construction can be found in~\cite{Yoshida-1991}, in the particular case of triangulations and hyperbolic structures.

For future reference, we summarise the above discussion in the following result.

\begin{lem}
	\label{lem:geometric_realisation_CR_structure}
	Let $\{ \phi_i, G_j \}$ be a geometric realisation of $\cell_f$ in $\overline{\H}$. If the geometric holonomy around each edge is trivial, then $\{ \phi_i, G_j \}$ defines a branched CR structure on $M_f$.
\end{lem}

In a similar fashion to ideal triangulations, the ideal cell decomposition $\cell_f$ we are going to construct is the complement of the $0$--skeleton of a CW complex, which is also called $\cell_f$. This CW complex is topologically homeomorphic to the \emph{end-compactification} of $M_f$. It has a single vertex, which is the only non-manifold point. When talking about (ideal) cells in $\cell_f$, it will be convenient to consider the $0$--skeleton as a point of reference, but we will not always underline that it is not actually part of the decomposition of $M_f$. Moreover, we are often going to drop the word ``ideal'' when it is clear from the context.\\

A \emph{finite geometric realisation} of $\cell_f$ in $\overline{\H}$ is a geometric realisation $\{ \phi_i, G_j \}$ whose embeddings $\phi_i : \sigma_i \rightarrow \overline{\H}$ extend to the $0$--skeleton. Finite geometric realisations are slightly easier to deal with, as we can use the image of the $0$--skeleton as reference points for the cells. Let $\widetilde{\cell_f}$ be the ideal cell decomposition of the universal cover $\widetilde{M_f}$ induced by $\cell_f$. If $\{ \phi_i, G_j \}$ is a finite geometric realisation with trivial geometric holonomy around each edge, then it defines a branched CR structure, represented by some pair $(\dev,\hol)$ of developing map and holonomy representation. By finiteness, the developing map $\dev : \widetilde{\cell_f} \rightarrow \overline{\H}$ extends equivariantly to the $0$--skeleton $\widetilde{\cell_f}^{(0)}$. More precisely, if $\dev^{(0)}$ is the restriction of $\dev$ to $\widetilde{\cell_f}^{(0)}$, then
\begin{equation*}
\hol(\gamma) \cdot \dev^{(0)}(x) = \dev^{(0)}( \gamma \cdot x), \qquad \gamma \in \pi_1(M_f), \quad x \in \widetilde{\cell_f}^{(0)}.
\end{equation*}

%%%%%%%%%%%%%%%%%%%%%%%%%%%%%

\subsection{The figure eight knot complement}
\label{subsec:figure_eight_complement}

The \emph{figure eight knot complement} $K_8$ is the $3$--manifolds obtained by removing a closed tubular neighbourhood of the figure eight knot from the three-sphere. Topologically, it is homeomorphic to the once-punctured torus bundle associated to the flip sequence $\bw_8 = \R \L$. The corresponding monodromy ideal triangulation $\tri_8$ has two tetrahedra: $\sigma_0^{\R}$ of type $\R$ and $\sigma_1^{\L}$ of type $\L$ (see Figure~\ref{fig:m004_RL}). As a cyclic word, $\bw_8$ has a subsequence $\L \R \L$ and a subsequence $\R \L \R$, corresponding to the two edges $e_{\R}$ and $e_{\L}$ of $\tri_8$ respectively (cf. Lemma~\ref{lem:edge_ribbon}). Both edges have valence six. The ribbon of tetrahedra around $e_{\R}$ is $\sigma_0^{\R}, \sigma_1^{\L}, \sigma_0^{\R}, \sigma_1^{\L}$, as depicted in Figure~\ref{fig:m004_around_edge}.\\

\begin{minipage}[ht]{0.55\textwidth}
	\centering
	\includegraphics[width=\textwidth]{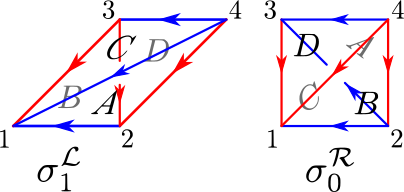}
	\captionof{figure}{The monodromy ideal triangulation of the figure eight knot complement $K_8$.}
	\label{fig:m004_RL}
\end{minipage}\hspace{1cm}
\begin{minipage}[ht]{0.35\textwidth}
	\centering
	\includegraphics[width=0.7\textwidth]{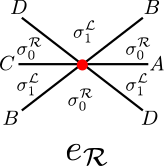}
	\captionof{figure}{The ribbon of tetrahedra around the red edge $e_{\R}$, viewed from the vertex $\sigma_0^{\R}(4)$.}
	\label{fig:m004_around_edge}
\end{minipage}

Now we construct a branched CR structure on $K_8$, as a preliminary example for the general case in~\S\ref{subsec:general_case}. The structure we are going to describe here was first discovered by Falbel~\cite{Falbel-2008}.\\

Let $\cell_8$ be the cell decomposition obtained from the following manipulations on the triangulation $\tri_8$.
\begin{enumerate}[label=\textnormal{(\arabic*)}]
	\item \label{item:decomp_type_R} (Figure~\ref{fig:decomposition_type_R}) We subdivide the face $\sigma_0^{\R}(134)$ of the tetrahedron $\sigma_0^{\R}$ into two $2$--cells, by introducing a $1$--cell with endpoints $\left\{\sigma_0^{\R}(1), \sigma_0^{\R}(4) \right\}$. The two $2$--cells thus obtained are combinatorially a triangle and a bigon. Similarly, we subdivide $\sigma_0^{\R}(234)$ by placing a $1$--cell with endpoints $\left\{\sigma_0^{\R}(2), \sigma_0^{\R}(4)\right\}$. Finally, we split the tetrahedron $\sigma_0^{\R}$ into two $3$--cells, by introducing a triangular $2$--cell with endpoints $\left\{\sigma_0^{\R}(1), \sigma_0^{\R}(2), \sigma_0^{\R}(4)\right\}$. Whence $\sigma_0^{\R}$ is subdivided into two $3$--cells:  $\sigmah_0^{\R}$ with vertices $\left\{\sigmah_0^{\R}(1), \sigmah_0^{\R}(2), \sigmah_0^{\R}(3), \sigmah_0^{\R}(4)\right\}$ is combinatorially isomorphic to a simplex, and $\sigmah_0^S$ with vertices $\left\{ \sigmah_0^S(1), \sigmah_0^S(2), \sigmah_0^S(4) \right\}$ is of the combinatorial type of a slab (cf.~\S\ref{subsec:CR_tetrahedra_slabs}).
	
	\item \label{item:decomp_type_L_1} (Figure~\ref{fig:decomposition_type_L}) Similar to above, we subdivide $\sigma_1^{\L}$ into two $3$--cells by introducing
	a $2$--cell inside the tetrahedron bounded by two $1$--cells with endpoints $\left\{\sigma_1^{\L}(2), \sigma_1^{\L}(4)\right\}$. They are embedded in the faces $\sigma_1^{\L}(124)$ and $\sigma_1^{\L}(234)$ respectively. Thus $\sigma_1^{\L}$ is decomposed into two $3$--cells $\sigmah_1^{\L} \cup \sigmah_1^W$. The former, $\sigmah_1^{\L}$ has four triangular faces and a bigon. The latter $\sigmah_1^W$ is of the combinatorial type of a \emph{wedge}, the CW complex obtained by quotienting a face of a $3$--simplex to a point. Its set of vertices is $\left\{\sigmah_0^W(2), \sigmah_0^W(4)\right\}$.
	
	\item \label{item:decomp_type_L_2} (Figure~\ref{fig:decomposition_type_L}) We deformation retract the wedge $\sigmah_1^W$ onto the bigonal face bounded by the red and the black edge. Simultaneously, we collapse the bigonal face of $\sigmah_1^{\L}$ into the black edge, transforming $\sigmah_1^{\L}$ back into a $3$--simplex. Finally, we remove the retracted wedge from the decomposition. As a consequence, the green edge and the black edge of $\sigmah_0^S$ are now identified (cf.~Figure~\ref{fig:decomposition_type_R} and Figure~\ref{fig:m004_cell_decomposition}).
\end{enumerate}

A few remarks are in order. Up to step~\ref{item:decomp_type_L_1}, the subdivisions of $\sigma_0^{\R}$ and $\sigma_1^{\L}$ agree along the faces, hence they form a well defined cell decomposition of $\tri_8$. The importance of this step relies on the fact that the new cell decomposition has more edges than $\tri_8$, hence a larger set where we can possibly branch on. On step~\ref{item:decomp_type_L_2}, we flatten the $3$--cell $\sigmah_1^W$ and remove it. This does not change the topology of the complex because a neighbourhood of the red edge $e_{\R}$ contains other $3$--cells other than $\sigmah_1^W$. In the end we have three $3$--cells $\sigmah_0^{\R}, \sigmah_0^S, \sigmah_1^{\L}$, two of which are of the combinatorial type of a tetrahedron and one of which is a slab (see Figure~\ref{fig:m004_cell_decomposition}). They glue to form a CW complex $\cell_8$, which is a cell decomposition of $K_8$. Step~\ref{item:decomp_type_L_2} is crucial because, by removing the wedge $\sigmah_1^W$ from the decomposition, we avoid the problem of having to geometrically realise it in CR space by an embedding. We remark that in~\cite{Falbel-2008}, Falbel develops this wedge into a flat bigon.\\

\begin{figure}[!t]
	\centering
	\includegraphics[height=3cm]{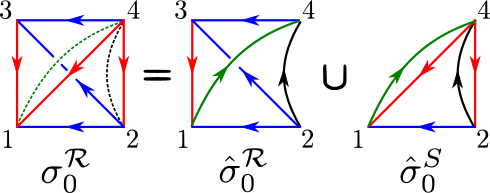}
	\captionof{figure}{The tetrahedron $\sigma_0^{\R}$ is subdivided into two $3$--dimensional cells, of the combinatorial type of a tetrahedron $\sigmah_0^{\R}$ and a slab $\sigmah_0^S$.}
	\label{fig:decomposition_type_R}
\end{figure}

\begin{figure}[!t]
	\centering
	\includegraphics[height=3cm]{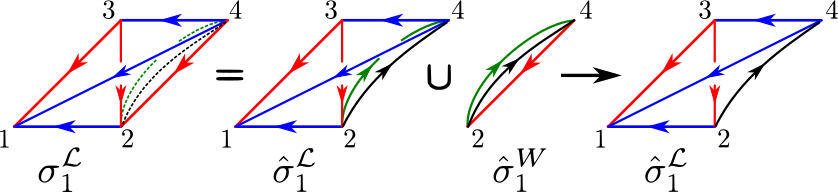}
	\captionof{figure}{The tetrahedron $\sigma_1^{\L}$ is decomposed into two $3$--cells, one of wich is a wedge $\sigmah_1^W$. The wedge is collapse and removed, while the other $3$--cell is deformed back into a tetrahedron $\sigmah_1^{\L}$.}
	\label{fig:decomposition_type_L}
\end{figure}

\begin{figure}[!t]
	\centering
	\includegraphics[height=3cm]{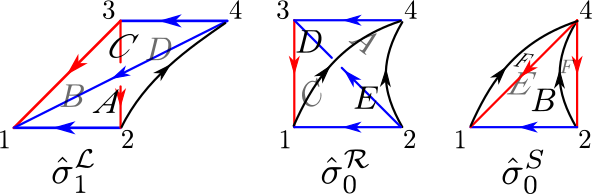}
	\captionof{figure}{The cell decomposition $\cell_8$ of the figure eight knot complement $K_8$.}
	\label{fig:m004_cell_decomposition}
\end{figure}

The slab $\sigmah_0^S$ has two bigonal faces, with endpoints $\left\{ \sigmah_0^S(1), \sigmah_0^S(2) \right\}$ and $\left\{ \sigmah_0^S(2), \sigmah_0^S(4) \right\}$. Since it would be ambiguous to refer to the edges of $\sigmah_0^S$ by their vertices, we fix the convention that $\sigmah_0^S(14)$ and $\sigmah_0^S(24)$ are the edges belonging to the face shared with $\sigmah_0^{\R}$, while $\sigmah_0^S(41)$ and $\sigmah_0^S(42)$ are the others. We will say more about these choices below.\\

We consider the following finite geometric realisation of $\cell_8$ in $\overline{\H}$. Let $\tetA, \tetB$ and $\slab_k$ be the two standard symmetric tetrahedra and the slab defined in~\S\ref{subsec:CR_tetrahedra_slabs}. The geometric realisations of the ideal cells are the combinatorial isomorphisms defined by
\begin{alignat*}{3}
& \phi_0^{\R} : \sigmah_0^{\R} \rightarrow \tetA, \qquad \qquad && \phi_0^S :  \sigmah_0^S \rightarrow \slab_1 \qquad \qquad \text{ and } \qquad && \phi_1^{\L} : \sigmah_1^{\L} \rightarrow \tetB\\[0.25cm]
& \phi_0^{\R}\left( \sigmah_0^{\R}(1) \right) := P_1, && \phi_0^S\left( \sigmah_0^S(1) \right) := P_1, && \phi_1^{\L}\left( \sigmah_1^{\L}(1) \right) := P_1,\\
& \phi_0^{\R}\left( \sigmah_0^{\R}(2) \right) := P_2, && \phi_0^S\left( \sigmah_0^S(2) \right) := P_2, && \phi_1^{\L}\left( \sigmah_1^{\L}(2) \right) := P_2,\\
& \phi_0^{\R}\left( \sigmah_0^{\R}(3) \right) := P_3, && \phi_0^S\left( \sigmah_0^S(4) \right) := P_4, && \phi_1^{\L}\left( \sigmah_1^{\L}(3) \right) := P_3,\\
& \phi_0^{\R}\left( \sigmah_0^{\R}(4) \right) := P_4, &&  && \phi_1^{\L}\left( \sigmah_1^{\L}(4) \right) := P_4.\\
\end{alignat*}
We remark that $\phi_0^S$ maps the edges $\sigmah_0^S(14)$ and $\sigmah_0^S(24)$ to the segments of $\CC$--circles going from $P_1$ and $P_2$, respectively, to $P_4$. Similarly, $\sigmah_0^S(41)$ and $\sigmah_0^S(42)$ are mapped to the segments of $\CC$--circles going from $P_4$ to $P_1$ and $P_2$, respectively.\\

The geometric realisations of the face pairings depicted in Figure~\ref{fig:m004_cell_decomposition} are the matrices $G_j$ defined in~\S\ref{subsec:CR_tetrahedra_slabs}, the identity matrix $I$ and a combination thereof. More precisely,
\begin{alignat*}{3}
& A : \sigmah_1^{\L}(124) \rightarrow \sigmah_0^{\R}(324) \quad && \text{ is realised by } \quad &&G_2 : \tetB \rightarrow \tetA, \\
& B : \sigmah_0^S(124) \rightarrow \sigmah_1^{\L}(123) && \text{ is realised by } && G_4 : \slab_1 \rightarrow \tetB, \\
& C : \sigmah_1^{\L}(134) \rightarrow \sigmah_0^{\R}(132) && \text{ is realised by } && G_1 : \tetB \rightarrow \tetA, \\
& D : \sigmah_0^{\R}(134) \rightarrow \sigmah_1^{\L}(234) && \text{ is realised by } && G_3 : \tetA \rightarrow \tetB, \\
& E : \sigmah_0^{\R}(124) \rightarrow \sigmah_0^S(124) && \text{ is realised by } && I : \tetA \rightarrow \slab_1, \\
& F :
\begin{array}{c}
\sigmah_0^S(14) \rightarrow \sigmah_0^S(24)\\
\sigmah_0^S(41) \rightarrow \sigmah_0^S(42)
\end{array}
&& \text{ is realised by } && G_2 G_3 : \slab_1 \rightarrow \slab_1. \\
\end{alignat*}
The product $G_2 G_3$, namely the geometric realisation of $F$, maps the bigonal face $B'$ of $\slab_1$ to its other bigonal face $B_1$. The combinatorics of $\cell_8$ around the red $e_{\R}$, black $e_{\R}'$ and blue $e_{\L}$ edges are displayed in Figure~\ref{fig:m004_edge_holonomy}. One computes that the geometric holonomies are trivial:
$$
e_{\R}: \ \left( G_2G_3 \right)^{-1} G_4^{-1}G_3G_1G_4 = I, \qquad e_{\R}' : \ G_3I^{-1}\left( G_2G_3 \right)^{-1} I G_2 = I,
$$ \vspace{-0.5cm}
$$
e_{\L} : \ G_1^{-1}IG_4^{-1}G_2^{-1}G_1G_3G_2 = I.
$$

\begin{figure}[ht]
	\centering
	\includegraphics[height=3cm]{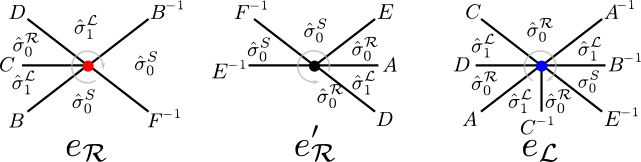}
	\captionof{figure}{The combinatorics around the red $e_{\R}$, black $e_{\R}'$ and blue $e_{\L}$ edges. The view is from the vertices $\sigmah_0^S(4), \sigmah_0^{\R}(4)$ and $\sigmah_0^{\L}(4)$ respectively.}
	\label{fig:m004_edge_holonomy}
\end{figure}

As per Lemma~\ref{lem:geometric_realisation_CR_structure}, this finite geometric realisation of $\cell_8$ in $\overline{\H}$ corresponds to a branched CR structure on $K_8$. By developing the cells in $\overline{\H}$, one finds that the order of the branching around the edges $e_{\R}$ and $e_{\R}'$ is one, while it is two around $e_{\L}$. These ramification orders were stated incorrectly in~\cite{Falbel-2008}, and corrected later in~\cite[Remark~6.1]{Falbel-Wang-2014}.\\

%%%%%%%%%%%%%%%%%%%%%%%%%%%%%

\subsection{General case}
\label{subsec:general_case}

Now we focus on the general case, to show that every hyperbolic once-punctured torus bundle $M_f$ admits a branched CR structure. In particular, we construct an ideal cell decomposition $\cell_f$ of $M_f$, and a finite geometric realisation of it in $\overline{\H}$, with trivial geometric holonomy around each edge.

\textbf{The ideal cell decomposition.} Let $f$ be an automorphism of the once punctured torus with two distinct positive real eigenvalues, and let $M_f$ be the corresponding hyperbolic once-punctured torus bundle. Suppose the flip sequence $\bw_f$ of $M_f$ has length $m$. Then the monodromy ideal triangulation $\tri_f$ of $M_f$ is made up of $m$ ideal tetrahedra $\sigma_0,\dots,\sigma_{m-1}$. The ideal cell decomposition $\cell_f$ of $M_f$ is obtained from $\tri_f$ by performing the three manipulations described in~\S\ref{subsec:figure_eight_complement} to each tetrahedron. We recall from~\S\ref{subsec:the_triangulation} that a tetrahedron is said to be of type $\R$ (resp. type $\L$) if the next tetrahedron is layered by a right (resp. left) layering. Thus we modify every tetrahedron of type $\R$ as in step~\ref{item:decomp_type_R}, and every tetrahedron of type $\L$ as in~\ref{item:decomp_type_L_1} and~\ref{item:decomp_type_L_2}. We provide a synthesis of those operations to refresh the notation.

\begin{enumerate}[label=\textnormal{(\arabic*)}]
	\item \label{item:summary_decomp_type_R} Every tetrahedron $\sigma_j^{\R}$ of type $\R$ is subdivided into two $3$--cells, along a newly introduced triangular $2$--cell with vertices $\left\{\sigma_j^{\R}(1), \sigma_j^{\R}(2), \sigma_j^{\R}(4)\right\}$. They are a tetrahedron $\sigmah_j^{\R}$ and a slab $\sigmah_j^S$.
	
	\item \label{item:summary_decomp_type_L_1} Every tetrahedron $\sigma_j^{\L}$ of type $\L$ is decomposed into two $3$--cells $\sigmah_j^{\L} \cup \sigmah_j^W$. The former $\sigmah_j^{\L}$ has four triangular faces, and a bigon where the wedge $\sigmah_j^W$ glues to.
	
	\item \label{item:summary_decomp_type_L_2} We deformation retract the wedge $\sigmah_j^W$ onto a bigonal face, then remove it. Simultaneously, we collapse the bigonal face of $\sigmah_j^{\L}$ into one edge, transforming $\sigmah_j^{\L}$ back into a $3$--simplex.
\end{enumerate}

Up to step~\ref{item:summary_decomp_type_L_1}, it is easy to check that the performed subdivisions agree along the faces of $\tri_f$, hence they form a well defined cell decomposition of $M_f$. 

Now consider the wedge $\sigma_j^W$. We claim that around each of its edges there is always at least one $3$--cell that is not a wedge. This is clear for two of its edges, as it glues to the tetrahedron $\sigma_j^T$. Call $e$ the remaining edge of $\sigma_j^W$. Let $\sigma_j^{\L}$ be the simplex of $\tri_f$ from which $\sigma_j^W$ is obtained, and let $\sigma_{j+1}^{\star}$ be the next tetrahedron that left layers on top of $\sigma_j^{\L}$. If $\sigma_{j+1}^{\star} = \sigma_{j+1}^{\L}$ is of type $\L$, then $\sigma_j^W$ glues to the wedge $\sigma_{j+1}^W$ around $e$. On the other hand, if $\sigma_{j+1}^{\star} = \sigma_{j+1}^{\R}$ is of type $\R$, then $\sigma_j^W$ glues to the slab $\sigma_{j+1}^S$ around $e$. Because $f$ has two distinct real eigenvalues, its flip sequence always contains at least one $\R$ and one $\L$ (cf.~\S\ref{subsec:flip_sequence}). It follows that around $e$ there is always at least one slab. This ends the proof of the claim.

On step~\ref{item:summary_decomp_type_L_2}, we flatten the wedges and remove them. It is a consequence of the claim that this does not change the topology of the complex. Thus in the end we have a CW complex $\cell_f$, consisting of three types of $3$--cells, two of which are of the combinatorial type of a tetrahedron and one of which is a slab. The complement of the $0$--skeleton is an ideal cell decomposition of $M_f$. We remark that $M_f$ has a more edges than $\cell_f$, and they are all going to be (non--trivially) branched (cf. \S\ref{subsec:branch_locus}).\\

To avoid introducing new terminology, we are going to make the following abuse of notation. Cells of $\cell_f$ coming from tetrahedra of $\tri_f$ of type $\R$ (resp. type $\L$) will also be referred to as cells of type $\R$ (resp. type $\L$). Moreover, if a tetrahedron $\sigma_j$ right layers (resp. left layers) on a tetrahedron $\sigma_{j-1}$ in $\tri_f$, then also the $3$--cells of $\cell_f$ obtained from $\sigma_j$ right layer (resp. left layer) on the cells obtained from $\sigma_{j-1}$.

\textbf{Combinatorics around the edges.} As mentioned in the example of figure eight knot complement, a slab $\sigmah_j^S$ has two bigonal faces, therefore it is ambiguous to refer to its edges by the $0$--skeleton. We avoid that by fixing the convention that $\sigmah_j^S(14)$ and $\sigmah_j^S(24)$ are the edges belonging to the face shared with the tetrahedron $\sigmah_j^\star$, $\star \in \{\L, \R\}$, while $\sigmah_j^S(41)$ and $\sigmah_j^S(42)$ are the others. The notation is motivated by the natural orientations of the edges of a geometric slab $\slab_k \subset \overline{\H}$.\\

Recall that $\pi$ is the natural quotient map from the disjoint union of the $m$ simplices of $\tri_f$ into $\tri_f$, defined by the face pairings. Let $\pih$ be the corresponding map for $\cell_f$. Then the \emph{valence} of an edge in $\cell_f$ is the size of its inverse image under $\pih$.

\begin{thm}
	\label{thm:edge_bijection_with_cells}
	Let $\cell_f^{(1)}$ be the set of $1$--cells in $\cell_f$. Let $A \subset \{0,\dots,m-1\}$ be the subset of indices such that $\sigmah_j^S$ is a slab of $\cell_f$. Then the quotient map $\pih$ restricts to a bijection
	$$
	\pih_r : \left\{ \sigmah_j^\star(14) \right\}_{j \in \{0\dots m-1\}} \cup \left\{ \sigmah_j^S(41) \right\}_{j \in A} \longrightarrow \cell_f^{(1)}, \qquad \star \in \{\L,\R \}.
	$$
\end{thm}

Theorem~\ref{thm:edge_bijection_with_cells} allows us to canonically pick a representative for each edge in $\cell_f$. For example, in the case of the figure eight knot complement in~\S\ref{subsec:figure_eight_complement}, the chosen representatives are $\sigmah_1^\L(14),\sigmah_0^\R(14)$ and $\sigmah_0^S(41)$ (respectively the blue, black and red edge in Figure~\ref{fig:m004_cell_decomposition}). Its proof is a consequence of the following two Lemmas, where we deduce the valence of edges in $\cell_f$ from their counterparts in $\tri_f$.

\begin{lem}
	\label{lem:edge_bijection_L}
	Let $\sigmah_j^{\L}$ be a $3$--cell of type $\L$ in $\cell_f$, corresponding to a tetrahedron $\sigma_j^\L$ in $\tri_f$. Let $2n_j +4$ be the valence of $\pi \left(\sigma_j^\L(14)\right)$. Then the equivalence class of $\sigmah_j^{\L}(14)$ in $\cell_f$ is
	{\fontsize{10}{11}\selectfont
		$$
		\left\{ \sigmah_j^{\L}(14), \ \left\{ \sigmah_{j+k}^{\R}(12), \ \sigma_{j+k}^{S}(12), \ \sigmah_{j+k}^{\R}(34)\right\}_ {k = 1,\dots,n_j}, \ \sigma_{j+n_j+1}^{\L}(12), \ \sigma_{j+n_j+1}^{\L}(34), \ \sigma_{j+n_j+2}^{\star}(23) \right\},
		$$	
	}
	where $\star \in \{\L,\R \}$. In particular $\pih \left(\sigmah_j^{\L}(14)\right)$ has valence $3n_j + 4$.
\end{lem}

\begin{proof}
	By Lemma~\ref{lem:edge_ribbon}, the edge $\sigma_j^\L(14)$ corresponds to a unique subsequence $\L \R^{n_j} \L$ of $\bw_f$, for $n_j \geq 0$. In particular, $\sigma_j^\L$ is the bottom of a unique ribbon of tetrahedra
	$$
	\sigma_j^{\L} \ \sigma_{j+1}^{\R} \  \cdots  \ \sigma_{j+n_j}^{\R}  \ \sigma_{j+n_j+1}^{\L} \ \sigma_{j+n_j+2}^{\star} \ ,
	$$
	where $\star \in \{ \L,\R \}$ is undetermined. Whence $\sigma_j^\L(14)$ is identified with the edges
	$$
	\sigma_j^{\L}(14), \ \left\{ \sigma_{j+k}^{\R}(12), \ \sigma_{j+k}^{\R}(34)\right\}_ {k = 1,\dots,n_j}, \ \sigma_{j+n_j+1}^{\L}(12), \ \sigma_{j+n_j+1}^{\L}(34), \ \sigma_{j+n_j+2}^{\star}(23).
	$$
	The valence of its equivalence class in $\tri_f$ is $2n_j + 4$. In $\cell_f$, we introduce a slab around each edge $\sigma_{j+k}^{\R}(12)$, while neighbourhoods of the other edges glued to $e_j$ are unchanged (cf.~Figure~\ref{fig:edge_cell_decomp_L} for $j=0$). The statement of the Lemma follows.
\end{proof}

\begin{lem}
	\label{lem:edge_bijection_R}
	Let $\sigmah_j^{\R}$ and $\sigmah_j^S$ be $3$--cells of type $\R$ in $\cell_f$, corresponding to a tetrahedron $\sigma_j^\R$ in $\tri_f$. Let $2n_j +4$ be the valence of $\pi \left(\sigma_j^\R(14)\right)$. Then the equivalence class of $\sigmah_j^{\R}(14)$ in $\cell_f$ is
	{\fontsize{11}{11}\selectfont
		$$
		\left\{\sigmah_j^{S}(14), \ \sigmah_j^{\R}(14), \ \left\{ \sigmah_{j+k}^{\L}(24)\right\}_ {k = 1,\dots,n_j}, \ \sigmah_{j+n_j+1}^{\R}(24), \ \sigmah_{j+n_j+1}^{S}(24)\right\}.
		$$
	}
	Similarly, the equivalence class of $\sigmah_j^S(41)$ in $\cell_f$ is
	{\fontsize{11}{11}\selectfont
		$$
		\left\{ \sigmah_j^{S}(41), \ \left\{ \sigmah_{j+k}^{\L}(13)\right\}_ {k = 1,\dots,n_j}, \ \sigmah_{j+n_j+1}^{\R}(42), \ \sigmah_{j+n_j+2}^{\star}(23), \ \sigmah_{j+n_j+1}^{S}(42) \right\}.
		$$
	}
	In particular, both $\pih \left(\sigmah_j^{\R}(14)\right)$ and $\pih \left(\sigmah_j^{\R}(14)\right)$ have valence $n_j + 4$.
\end{lem}

\begin{proof}
	As in the proof of Lemma~\ref{lem:edge_bijection_L}, the edge $\sigma_j^\R(14)$ corresponds to a unique  subsequence $\R \L^{n_j} \R$ in $\bw_f$, for $n_j \geq 0$. The ribbon of tetrahedra around its edge class in $\tri_f$ is
	$$
	\sigma_j^{\R} \ \sigma_{j+1}^{\L} \  \cdots  \ \sigma_{j+n_j}^{\L}  \ \sigma_{j+n_j+1}^{\R} \ \sigma_{j+n_j+2}^{\star}.
	$$
	In particular $\sigma_j^\R(14)$ is glued to the $2n_j + 4$ edges
	$$
	\sigma_j^{\R}(14), \ \left\{ \sigma_{j+k}^{\L}(13), \ \sigma_{j+k}^{\L}(24)\right\}_ {k = 1,\dots,n_j}, \ \sigma_{j+n_j+1}^{\R}(13), \ \sigma_{j+n_j+1}^{\R}(24), \ \sigma_{j+n_j+2}^{\star}(23).
	$$
	In $\cell_f$, the cell $\sigma_j^\R(14)$ splits into the bigon with boundary $\sigmah_j^S(14)$ and $\sigmah_j^S(41)$. The two loops of the ribbon of tetrahedra around $\sigma_j^\R(14)$ are split and equidistributed around those two edges (cf.~Figure~\ref{fig:edge_cell_decomp_R} for $j=0$). The statement of the Lemma follows.
\end{proof}

\begin{figure}[!t]
	\begin{minipage}[t]{0.45\textwidth}
		\centering
		\includegraphics[height=4cm]{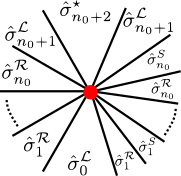}
		\captionof{figure}{The cross section of a neighbourhood of $\pih\left(\sigmah_0^{\L}(14)\right)$ in $\cell_f$, viewed from the vertex $\sigmah_0^{\L}(4)$.}
		\label{fig:edge_cell_decomp_L}
	\end{minipage} \hspace{1cm}
	\begin{minipage}[t]{0.45\textwidth}
		\centering
		\includegraphics[height=4cm]{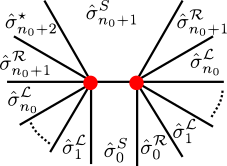}
		\captionof{figure}{The edge splits into two edges, $\pih \left(\sigmah_0^S(41)\right)$ on the left and $\pih \left(\sigmah_0^S(14)\right)$ on the right. The view is from the vertex $\sigmah_0^S(4)$}
		\label{fig:edge_cell_decomp_R}
	\end{minipage}
\end{figure}

\begin{proof}[Proof of Theorem~\ref{thm:edge_bijection_with_cells}]
	First we notice that $\pih_r$ is well defined, as it is the restriction of the natural quotient map $\pih$. Injectivity follows from Lemma~\ref{lem:edge_bijection_L} and Lemma~\ref{lem:edge_bijection_R}, because the equivalence classes of $\sigmah_j^{\L}(14),\sigmah_j^{\R}(14)$ and $\sigmah_j^S(41)$ are distinct.
	
	By a topological argument, we deduce that the Euler characteristic of $\cell_f$ is zero. Therefore $\cell_f$ has as many $3$--cells as $1$--cells. It follows that $\pih_r$ is an injective map between finite sets with the same sizes, thus it is a bijection.
\end{proof}

\textbf{The finite geometric realisation in $\overline{\H}$.} A finite geometric realisation of $\cell_f$ consists of embeddings $\phi_i^{\L},\phi_i^{\R}, \phi_i^S$ of the $3$--cells into $\overline{\H}$, and geometric realisations $G_j \in \PU(2,1)$ of the face pairings.\\

Let $\sigmah_j^\star$ be a tetrahedron of $\cell_f$, $\star \in \{ \L,\R \}$. The development of $\sigmah_j^\star$ depends on the tetrahedron it layers on. More precisely, let $\sigmah_{j-1}$ be the tetrahedron in $\cell_f$ on top of which $\sigmah_j^\star$ layers. Then the geometric realisation $\phi_j^\star$ of $\sigmah_j^\star$ is the combinatorial isomorphism
$$
\phi_j^\star :
\begin{cases}
\sigmah_j^\star \rightarrow \tetB  & \text{if } \ \sigmah_{j-1}  = \sigmah_{j-1}^\R  \text{ is of type } \R,\\
\sigmah_j^\star \rightarrow \tetA & \text{if } \ \sigmah_{j-1}  = \sigmah_{j-1}^\L  \text{ is of type } \L,
\end{cases}
\qquad
\text{where}
\qquad
\begin{array}{c}
\phi_j^\star\left( \sigmah_j^\star(1) \right) := P_1,\\[0.2cm]
\phi_j^\star\left( \sigmah_j^\star(2) \right) := P_2,\\[0.2cm]
\phi_j^\star\left( \sigmah_j^\star(3) \right) := P_3,\\[0.2cm]
\phi_j^\star\left( \sigmah_j^\star(4) \right) := P_4.
\end{array}
$$

Now let $\sigmah_j^S$ be a slab of $\cell_f$. Let $k_j +4$ be the valence of the edge $\pih \left(\sigmah_j^S(24)\right)$. Then the geometric realisation $\phi_j^S$ of $\sigmah_j^S$ is the combinatorial isomorphism
$$
\phi_j^S : \sigmah_j^S \rightarrow \slab_{k_j} \qquad
\text{where}
\qquad
\begin{array}{c}
\phi_j^S\left( \sigmah_j^S(1) \right) := P_1,\\[0.2cm]
\phi_j^S\left( \sigmah_j^S(2) \right) := P_2,\\[0.2cm]
\phi_j^S\left( \sigmah_j^S(4) \right) := P_4.
\end{array}
$$
More precisely, we require that $\phi_j^S\left( \sigmah_j^S(14) \right) = [P_1,P_4]$ and $\phi_j^S\left( \sigmah_j^S(24) \right) = [P_2,P_4]$. Thus the bigon with endpoints $\left\{ \sigmah_j^S(1),\sigmah_j^S(4)  \right\} $ is developed into
\begin{alignat*}{4}
&B' := &&\left( 1 + t e^{-i\frac{\pi}{6}}, s \right), \qquad &&t \in \RR_{>0} \cup \{\infty\}, \quad s \in \RR \cup \{\infty\},
\end{alignat*}
while the bigon with endpoints $\left\{ \sigmah_j^S(2),\sigmah_j^S(4)  \right\} $ is realised by
\begin{alignat*}{4}
&B_{k_j} := &&\left( -\omega + t e^{-i\frac{\pi}{6}(1-2{k_j})}, s \right), \qquad &&t \in \RR_{>0} \cup \{\infty\}, \quad s \in \RR \cup \{\infty\}.
\end{alignat*}
Both $B'$ and $B_k$ are foliated by vertical $\CC$--circles.\\

Most of the geometric realisations of the face pairings are uniquely determined by Lemma~\ref{lem:cartan_angle}. They are the CR transformations $G_i$ described in~\S\ref{subsec:CR_tetrahedra_slabs}. The remaining ones are either the identity matrix $I$, or products of the $G_i$'s. We describe them in more detail below. Let $\sigmah_j^\star$ be a tetrahedron of $\cell_f$, of type $\star \in \{ \L, \R \}$.

If $\phi_j^\star \left(\sigmah_j^\star\right) = \tetA$ is the standard symmetric tetrahedron of type $A$, then $\sigmah_j^\star$ layers on a tetrahedron $\sigmah_{j-1}^{\L}$ of type $\L$. In particular they share two pairs of faces. Let $\mathfrak{T}_X = \phi_{j-1}^{\L} \left(\sigmah_{j-1}^{\L}\right)$ for some $X \in \{A,B\}$. Then the geometric realisations of the face pairings between $\sigmah_{j-1}^{\L}$ and $\sigmah_j^\star$ are:
\begin{alignat*}{4}
& \sigmah_{j-1}^{\L}(134) \rightarrow \sigmah_j^{\star}(132) \quad && \text{ is realised by } \quad &&G_1 : \mathfrak{T}_X \rightarrow \tetA,\\
& \sigmah_{j-1}^{\L}(124) \rightarrow \sigmah_j^{\star}(324) \quad && \text{ is realised by } \quad &&G_2 : \mathfrak{T}_X \rightarrow \tetA.
\end{alignat*}

Now suppose $\phi_j^\star \left(\sigmah_j^\star\right) = \tetB$ is the standard symmetric tetrahedron of type $B$. In this case $\sigmah_j^\star$ layers on a tetrahedron $\sigmah_{j-1}^{\R}$ of type $\R$ and on a slab $\sigmah_{j-1}^{S}$. Let $\mathfrak{T}_X = \phi_{j-1}^{\R} \left(\sigmah_{j-1}^{\R}\right)$, for some $X \in \{A,B\}$, and let $\slab_{k_{j-1}} = \phi_{j-1}^S \left(\sigmah_{j-1}^S\right)$. Then the geometric realisations of the face pairings between $\sigmah_{j-1}^{\R}, \sigmah_{j-1}^{S}$ and $\sigmah_j^\star$ are
\begin{alignat*}{4}
& \sigmah_{j-1}^{\R}(134) \rightarrow \sigmah_j^{\star}(234) \quad && \text{ is realised by } \quad &&G_3 : \mathfrak{T}_X \rightarrow \tetB,\\
& \sigmah_{j-1}^S(124) \rightarrow \sigmah_j^{\star}(123) \quad && \text{ is realised by } \quad &&G_4 : \slab_{k_{j-1}} \rightarrow \tetB,\\
& \sigmah_{j-1}^{\R}(134) \rightarrow \sigmah_{j-1}^S(124) \quad && \text{ is realised by } \quad && I : \mathfrak{T}_X \rightarrow \slab_{k_{j-1}}.
\end{alignat*}

These cover all cases, except for the gluing maps between the bigonal faces of the slabs. Contrary to marked triangles, bigons in Heisenberg space can be identified via many CR transformations. Earlier in this section we showed that around each edge in $\cell_f$ there is at most one face pairing gluing two slabs along their bigons (cf.~Lemma~\ref{lem:edge_bijection_L} and Lemma~\ref{lem:edge_bijection_R}). Whence we are going to geometrically realise those face pairings so that the geometric holonomy around each edge is trivial. Under this condition, the choices turn out to be unique.

Consider the slab $\sigmah_{j}^S$. By Lemma~\ref{lem:edge_bijection_R}, the equivalence class of the edge $\sigmah_{j}^S(14)$ is
$$
\left\{\sigmah_j^{S}(14), \ \sigmah_j^{\R}(14), \ \left\{ \sigmah_{j+k}^{\L}(24)\right\}_ {k = 1,\dots,n_j}, \ \sigmah_{j+n_j+1}^{\R}(24), \ \sigmah_{j+n_j+1}^{S}(24)\right\}.
$$
Let $A_j\dots,A_{j+n_j+2}$ be the sequence of geometric realisations of the face pairings around $\pih \left(\sigmah_{j}^S(14)\right)$, starting from $\sigmah_j^{S}$ to $\sigmah_{j+n_j+1}^S$, travelling anticlockwise from the point of view of the vertex $\sigmah_{j}^S(4)$. So for example $A_j$ realises the face pairing between $\sigmah_{j}^S$ and $\sigmah_{j}^{\R}$, while $A_{j+n_j+2}$ corresponds to $\sigmah_{j+n_j+1}^{\R}$ and $\sigmah_{j+n_j+1}^S$ (cf.~Figure~\ref{fig:edge_cell_decomp_R} on the right). We remark that $\sigmah_{j+n_j+1}^S$ is geometrically realised by the slab $\slab_{n_j}$, because the edge $\pih\left( \sigmah_{j+n_j+1}^S(24)\right)$ has valence $n_j+4$.

\begin{lem}
	\label{lem:geometric_realisation_between_bigons}
	The matrix product $\prod_{k=0}^{n_j+2} A_{j+n_j+2-k}$ is a geometric realisation of the face pairing between $\sigmah_{j}^S$ and $\sigmah_{j+n_j+1}^S$. In particular, it identifies the bigon $B'$ of $\phi_j^S \left(\sigmah_{j}^S\right)$ with the bigon $B_{n_j}$ of $\phi_{j+n_j+1}^S \left(\sigmah_{j+n_j+1}^S\right)$.
\end{lem}

\begin{proof}
	By construction, $A_j$ and $A_{j+n_j+2}$ are the identity matrix. On the other hand, $A_{j+1}~=~G_3$ and $A_{j+k} = G_2$, for all $k \in \{2,\dots,n_j+1\}$. Therefore
	$$
	\prod_{k=0}^{n_j+2} A_{j+n_j+2-k} = G_2^{n_j} G_3.
	$$
	We recall from~\S\ref{subsec:CR_tetrahedra_slabs} that the CR transformations $G_3$ and $G_2$ preserve vertical $\CC$--circles, and restrict to rotations on the $z$--plane. In particular, $G_3$ maps $B'$ to the bigon $B_0$ and $G_2$ maps $B_k$ to $B_{k+1}$. The Lemma follows.
\end{proof}

We remark that the face pairing $\prod_{k=0}^{n_j+2} A_{j+n_j+2-k}$ is monotone, thus this completes the construction of the finite geometric realisation of $\cell_f$. We conclude the section by showing that these geometric realisations are indeed branched CR structures.

\begin{thm}
	\label{thm:geometric_realisation_is_structure}
	The geometric holonomy around each edge in $\cell_f$ is trivial and therefore the geometric realisation defines a branched CR structure on $M_f$.
\end{thm}

\begin{proof}
	We recall that by Theorem~\ref{thm:edge_bijection_with_cells} there is a canonical representative for each edge in $\cell_f$.
	
	Let $A \subset \{0,\dots,m-1\}$ be the subset of indexes such that $\sigmah_j^S$ is a slab of $\cell_f$, and let $\overline{A} = \{0,\dots,m-1\} \setminus A$ be its complement. It is a consequence of Lemma~\ref{lem:geometric_realisation_between_bigons} that the geometric holonomy around the edges $\pih \left(\sigmah_{j}^{\R}(14)\right)$, for $j\in A$, is trivial.
	
	Consider an edge $\pih \left(\sigmah_{j}^{S}(41)\right)$, for $j\in A$. Let $A_j\dots,A_{j+n_j+3}$ be the sequence of geometric realisations of all the face pairings around $\pih \left(\sigmah_{j}^{S}(41)\right)$, starting from $\sigmah_j^{S}$ and travelling clockwise from the point of view of the vertex $\sigmah_{j}^S(4)$ (cf.~Figure~\ref{fig:edge_cell_decomp_R} on the left). Then we have
	\begin{alignat*}{2}
	&A_j = G_4, \qquad &&A_{j+k} = G_1 \quad \text{for} \quad k \in \{1,\dots,n_j\},\\ 
	&A_{j+n_j+1} = G_3, \qquad &&A_{j+n_j+2} = G_4^{-1} \  \qquad A_{j+n_j+3} = G_3G_2^{-n_j}. 
	\end{alignat*}
	Thus the geometric holonomy around $\pih \left(\sigmah_{j}^{S}(41)\right)$ is the product $G_3G_2^{-n_j}G_4^{-1}G_3G_1^{n_j}G_4$. Because the matrices $G_1$ and $G_2$ are of order six, one only needs to check that the product is the identity matrix for $n_j \in \{0,\dots,5\}$. Straight forward computation of the six products gives the result.\\
	
	An analogous argument works for the edges $\pih \left(\sigmah_{j}^{\L}(14)\right)$, $j\in \overline{A}$. The geometric holonomy around them is of the form $G_1^{-1}G_4^{-n_j}G_2^{-1}G_1G_3^{n_j}G_2$. The matrices $G_3$ and $G_4$ are also of order six, hence one only needs to check that the cases $n_j \in \{0,\dots,5\}$. The calculation is straightforward.
	
	We apply Lemma~\ref{lem:geometric_realisation_CR_structure} to complete the proof.
\end{proof}

%%%%%%%%%%%%%%%%%%%%%%%%%%%%%
%%%%%%%%%%%%%%%%%%%%%%%%%%%%%

\section{Properties of the Structures}
\label{sec:property_structures}

Consider the branched CR structure on the hyperbolic once-punctured torus bundle $M_f$ described in the previous section~\ref{subsec:general_case}. We conclude by analysing two important features of the structure: the ramification order around each connected component of the branch locus (namely the ideal edges), and the holonomy representation. In~\S\ref{subsec:branch_locus} we show that the ramification order of an edge $e$ has a simple description in terms of the valence of $e$ in the cell decomposition, and therefore its explicitly computable (cf.~Theorem~\ref{thm:edge_bijection_with_cells}). In~\S\ref{subsec:holonomy} we find the holonomy of the generators of $\pi_1(M_f)$ and underline some properties.

%%%%%%%%%%%%%%%%%%%%%%%%%%%%%

\subsection{Branch locus}
\label{subsec:branch_locus}

The branch locus of the CR structure of $M_f$ is set of all ideal edges of the associated cell decompositions $\cell_f$. Here we show that the ramification order around each curve is related to their valence in the simplicial complex. The strategy will be to develop each curve as a vertical line in Heisenberg space, and analyse the projection onto the $\CC$--plane of a neighbourhood. This way we can talk about angles of the projections where otherwise it would not be possible. We remind the reader that CR transformations do not preserve angles, therefore the angles we are going to talk about depend on the chosen realisations.\\

We recall that by Theorem~\ref{thm:edge_bijection_with_cells} there is a canonical representative for each edge in the cell decomposition $\cell_f$, namely $\sigmah_j^{\R}(14),\sigmah_j^S(41)$ and $\sigmah_j^{\L}(14)$. Let $\text{ceiling}(x) = \lceil x \rceil$ be the \emph{ceiling function}, which associates $x$ to the smallest integer greater than or equal to $x$.

\begin{lem}
	\label{lem:branching_around_edge_R}
	Let $n_j+4$ be the valence of $e_j = \pih \left( \sigmah_j^{\R}(14) \right)$ in $\cell_f$. Then the ramification order around $e_j$ is $\left \lceil{ \frac{n_j+5}{6} }\right \rceil$.
\end{lem}

\begin{proof}
	First, we observe that the geometric realisation $\phi_j^{\R}$ develops the edge $\sigmah_j^{\R}(14)$ into the vertical ray of Heisenberg space going from $P_1 = (1,\sqrt{3})$ to $P_4 = \infty$. Therefore we can understand the ramification order of $e_j$ by looking at the projections of the tetrahedra around $e_j$ on the $\CC$--plane of $\overline{\H}$. 
	
	Let $R \subset \CC$ be the projection of the standard symmetric tetrahedron. It is a triangular region bounded by three arcs of circles (cf.~Figure~\ref{fig:CR_tetrahedra}). We recall from~\S\ref{subsec:general_case} (cf. Figure~\ref{fig:edge_cell_decomp_R}) that the sequence of $3$--cells around $e_j$ in $\cell_f$ is
	$$
	\sigmah_j^S,\sigmah_j^{\R},\sigmah_{j+1}^{\L},\dots,\sigmah_{j+n_j}^{\L},\sigmah_{j+n_j+1}^{\R}, \sigmah_{j+n_j+1}^S.
	$$
	Then $\phi_j^{\R} \left(\sigmah_j^{\R}\right)$ projects onto $R_j := R$. The next simplex glues to $\phi_j^{\R} \left(\sigmah_j^{\R}\right)$ via $G_3^{-1}$, therefore its projection $R_{j+1}$ is a $\frac{\pi}{3}$ clockwise rotation of $R_j$ around the origin. After that, we have $n_j$ simplices each of which is glued to the previous one by $G_2^{-1}$. Whence each of their projections $R_{j+k}$, for $k\in \{1,\dots,n_j+1\}$, is a $\frac{\pi}{3}$ anticlockwise rotation of $R_{k-1}$ about the point $1$. Finally, the projections of the geometric realisations of the two slabs $\sigmah_j^S, \sigmah_{j+n_j+1}^S$ rigidly glue to $R_j$ and $R_{j+n_j+1}$ to fill in the gap. Examples for $n_j = 1$ and $n_j = 3$ are depicted in Figure~\ref{fig:branch_locus_R}.
	
	\begin{figure}[t!]
		\centering
		\includegraphics[height=4.3cm]{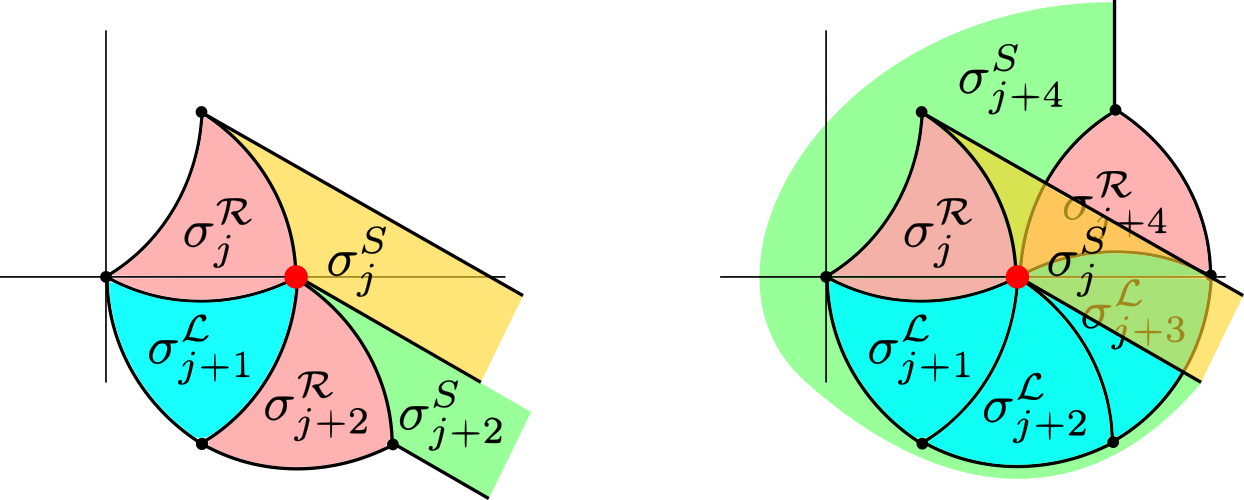}
		\captionof{figure}{The developments around the branch locus for $n_j = 1$ (on the left) and $n_j = 3$ (on the right). Their respective ramification orders are one and two.}
		\label{fig:branch_locus_R}
	\end{figure}
	
	Around $e_j$, the first region $R_j$ contributes with an angle of $\frac{2}{3}\pi$, while every other region $R_{j+k}$, for $k\in \{1,\dots,n_j+1\}$, contributes with an angle of $\frac{\pi}{3}$. The first slab also adds $\frac{2}{3}\pi$. This sums up to $\left(\frac{5+n_j}{6}\right) 2\pi$. The angle of the projection of the last slab around $e_j$ is a non-negative number strictly lower than $2 \pi$, therefore the total angle is the next integer multiple of $2\pi$. That is $\left \lceil{ \frac{n_j+5}{6} }\right \rceil 2\pi$.
\end{proof}

\begin{lem}
	\label{lem:branching_around_edge_bigon}
	Let $n_j+4$ be the valence of $e_j = \pih \left( \sigmah_j^S(41) \right)$ in $\cell_f$. Then the ramification order around $e_j$ is $\left \lceil{ \frac{n_j+5}{6} }\right \rceil$.
\end{lem}

\begin{proof}
	This proof is similar to the one of Lemma~\ref{lem:branching_around_edge_R}, as the geometric realisation $\phi_j^{S}$ develops the edge $\sigmah_j^S(41)$ into the vertical ray from $P_4 = \infty$ to $P_1 = (1,\sqrt{3})$. The only difference is that we are not going to consider the projections of the entire cells, since they are not as tidy as in the previous case, but only the projections of the vertices. Every $3$--cell around $e_j$ has two vertices at $P_4$ and $P_1$, and its angle about $e_j$ is strictly between zero and $2\pi$. Therefore knowing the positions of the other vertices gives us an estimate of the total angle around $e_j$.
	
	The sequence of $3$--cells around $e_j$ in $\cell_f$ is
	$$
	\sigmah_j^S,\sigmah_{j+1}^{\L},\dots,\sigmah_{j+n_j}^{\L},\sigmah_{j+n_j+1}^{\R},\sigmah_{j+n_j+2}^{\star}, \sigmah_{j+n_j+1}^S, 
	$$
	for some $\star \in \{ \L, \R \}$. We begin by developing $\phi_j^{S} \left(\sigmah_j^S\right)$, then glue every other $3$--cell around $e_j$. The vertices that are not identified with the endpoints of $e_j$ are listed in Table~\ref{table:vertices_around_edge_bigon}. They are all positioned at the vertices of a regular hexagon of edge length $\frac{\sqrt{3}}{2}$. 
	
	\begin{table}[t!]
		{\centering
			\begin{tabular}{|| >{\centering\arraybackslash} m{2.5cm} | >{\centering\arraybackslash} m{3cm} | >{\centering\arraybackslash} m{4cm} | >{\centering\arraybackslash} m{3cm} ||} 
				\hline
				CR face pairing & $3$--cells & Vertices disjoint from $e_j$ & $\CC$--coordinates \\
				\hline\hline
				\multirow{2}{2.7cm}{$G_4$} & \vspace{0.1cm} $\sigmah_j^S$ &  \vspace{0.1cm} $\sigmah_j^S(2)$ & \vspace{0.1cm} $-\omega$ \\  \cline{2-4}
				& \multirow{2}{1.5cm}{ $\sigmah_{j+1}^{\L}$} & \vspace{0.1cm} $\sigmah_{j+1}^{\L}(2)$ & \vspace{0.1cm} $-\omega$ \\ \cline{1-1} \cline{3-4}
				\multirow{2}{2.7cm}{$G_1$} &  & \vspace{0.1cm} $\sigmah_{j+1}^{\L}(4)$ & \vspace{0.1cm} $0$ \\ \cline{2-4}
				& \vspace{0.1cm} $\sigmah_{j+1+k}^{\L}$, &  \vspace{0.1cm} $\sigmah_{j+1+k}^{\L}(2)$ & \vspace{0.1cm} $(-1)^k \omega^{k-1} +1$ \\ \cline{1-1} \cline{3-4}
				\multirow{2}{2.7cm}{$G_1$} &  $k=1,\dots,n_j-1$ &  \vspace{0.1cm} $\sigmah_{j+1+k}^{\L}(4)$ & \vspace{0.1cm} $(-1)^{k+1} \omega^{k} +1$ \\ \cline{2-4}
				& \multirow{2}{1.5cm}{$\sigmah_{j+n_j+1}^{\R}$} &  \vspace{0.1cm} $\sigmah_{j+n_j+1}^{\R}(2)$ & \vspace{0.1cm} $(-1)^n_j \omega^{n_j-1} +1$ \\  \cline{1-1} \cline{3-4}
				\multirow{2}{2.7cm}{$G_3$} & & \vspace{0.1cm} $\sigmah_{j+n_j+1}^{\R}(4)$ & \vspace{0.1cm} $(-1)^{n_j+1} \omega^{n_j} +1$ \\ \cline{2-4}
				& \multirow{2}{1.5cm}{$\sigmah_{j+n_j+2}^{\star}$} &  \vspace{0.1cm} $\sigmah_{j+n_j+2}^{\star}(4)$ & \vspace{0.1cm} $(-1)^{n_j+1} \omega^{n_j} +1$ \\  \cline{1-1} \cline{3-4}
				\multirow{2}{2.7cm}{$G_4^{-1}$} & &  \vspace{0.1cm} $\sigmah_{j+n_j+2}^{\star}(1)$ & \vspace{0.1cm} $(-1)^{n_j+2} \omega^{n_j+1} +1$ \\ \cline{2-4}
				& \vspace{0.1cm} $\sigmah_{j+n_j+1}^S$ &  \vspace{0.1cm} $\sigmah_{j+n_j+1}^S(1)$ & \vspace{0.1cm} $(-1)^{n_j+2} \omega^{n_j+1} +1$ \\  \cline{1-1} \cline{3-4}
				\hline
				\hline
		\end{tabular}}
		\caption{The list of vertices of the $3$--cells around $e_j$ that are not identified with the endpoints of $e_j$. We recall that $\omega= -\frac{1}{2}\left( 1 + i \sqrt{3}\right)$.}
		\label{table:vertices_around_edge_bigon}
	\end{table}
	
	We draw examples of the projections for $n_j = 1$ and $n_j = 3$ in Figure~\ref{fig:branch_locus_L}. We remark that these are projections of the vertices and edges, but not of the $2$--skeletons as faces are generally not foliated by vertical rays anymore.
	
	\begin{figure}[t!]
		\centering
		\includegraphics[height=4.3cm]{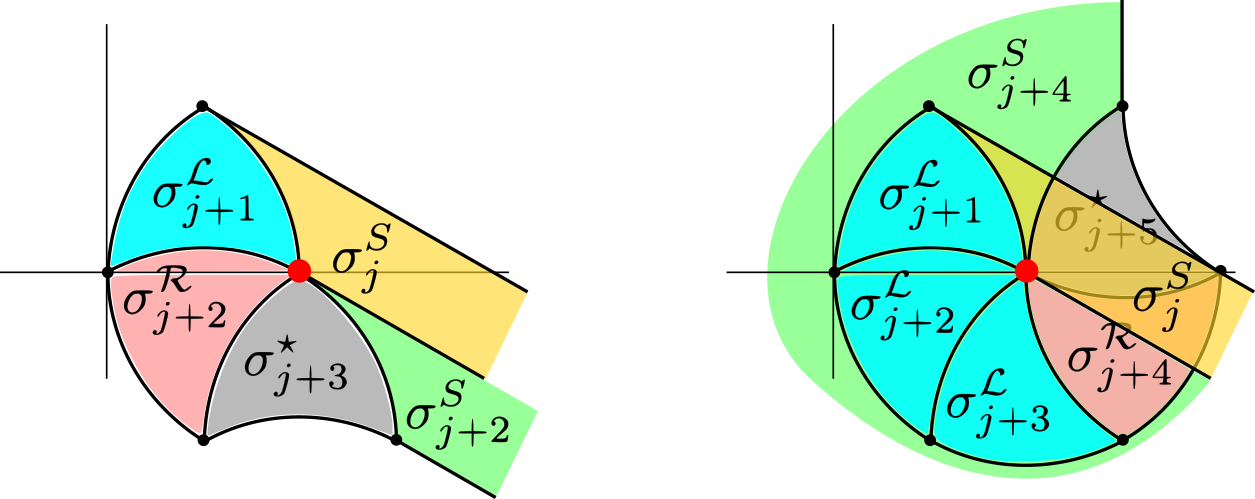}
		\captionof{figure}{The developments around the branch locus for $n_j = 1$ (on the left) and $n_j = 3$ (on the right). Only vertices and edges are projected. The shaded areas are just guidelines to distinguish the different cells, but they are not the actual projections of the $3$--cells. The respective ramification orders are one and two.}
		\label{fig:branch_locus_L}
	\end{figure}
	
	Up to $\sigmah_{j+n_j+2}^{\star}$, the total sum of the angles is strictly between $\left \lceil{ \frac{n_j-1}{6} }\right \rceil$ and $\left \lceil{ \frac{n_j+5}{6} }\right \rceil.$ Because the angle of the projection of the last slab around $e_j$ is a non-negative number strictly lower than $2 \pi$, the ramification order must be $\left \lceil{ \frac{n_j+5}{6} }\right \rceil$.
\end{proof}

\begin{lem}
	\label{lem:branching_around_edge_L}
	Let $3n_j+4$ be the valence of $e_j = \pih \left( \sigmah_j^{\L}(14) \right)$ in $\cell_f$. Then the ramification order around $e_j$ is $n_j+1$.
\end{lem}

\begin{proof}
	We follow almost verbatim the proof of Lemma~\ref{lem:branching_around_edge_bigon}.
	
	First we consider the development $\phi_j^{\L} \left(\sigmah_j^{\L}\right)$. This geometric realisation maps the edge $\sigmah_j^{\L}(14)$ into the vertical ray from $P_1 = (1,\sqrt{3})$ to $P_4 = \infty$. From the point of view of the vertex $\sigmah_j^{\L}(4)$ (cf.~Figure~\ref{fig:edge_cell_decomp_L}), starting from $\sigmah_j^{\L}$ and travelling anticlockwise around $e_j$ until $\sigma_{j+n_j+2}^{\star}$, we encounter the $3$--cells
	$$
	\sigma_j^{\L} \ \sigma_{j+1}^{\R} \ \sigma_{j+1}^{S} \  \cdots  \ \sigma_{j+n_j}^{\R}  \ \sigma_{j+n_j}^S  \ \sigma_{j+n_j+1}^{\L} \ \sigma_{j+n_j+2}^{\star}.
	$$
	The vertices of these cells that are not identified with the endpoints of $e_j$ are listed in Table~\ref{table:vertices_around_edge_L_1}.
	
	\begin{table}[!t]
		{\centering
			\begin{tabular}{|| >{\centering\arraybackslash} m{2.5cm} | >{\centering\arraybackslash} m{3cm} | >{\centering\arraybackslash} m{4cm} | >{\centering\arraybackslash} m{3cm} ||} 
				\hline
				CR face pairing & $3$--cells & Vertices disjoint from $e_j$ & $\CC$--coordinates \\
				\hline\hline
				\multirow{3}{2.7cm}{$G_1$} & \vspace{0.1cm} \multirow{2}{1.5cm}{ $\sigmah_{j}^{\L}$}  &  \vspace{0.1cm} $\sigmah_j^{\L}(2)$ & \vspace{0.1cm} $-\omega$ \\  \cline{3-4}
				& &  \vspace{0.1cm} $\sigmah_j^{\L}(3)$ & \vspace{0.1cm} $0$ \\  \cline{2-4}
				& \vspace{0.1cm} $\sigmah_{j+k}^{\R}$ & \vspace{0.1cm} $\sigmah_{j+k}^{\R}(3)$ & \vspace{0.1cm} $(-1)^{k}\overline{\omega}^{k-1}+1$ \\ \cline{1-1} \cline{3-4}
				\multirow{2}{2.7cm}{$I$} & $k=1,\dots,n_j$ & \vspace{0.1cm} $\sigmah_{j+k}^{\R}(4)$ & \vspace{0.1cm} $(-1)^{k+1}\overline{\omega}^{k}+1$ \\ \cline{2-4}
				& \vspace{0.1cm} $\sigmah_{j+k}^{S}$, &  \vspace{0.1cm} \multirow{2}{4cm}{\centering $\sigmah_{j+k}^{S}(4)$} & \vspace{0.1cm} \multirow{2}{3cm}{\centering $(-1)^{k+1}\overline{\omega}^{k}+1$}  \\ \cline{1-1}
				\multirow{2}{2.7cm}{$G_4$} &  $k=1,\dots,n_j$ &  & \\ \cline{2-4}
				& \multirow{2}{1.5cm}{$\sigmah_{j+n_j+1}^{\L}$} &  \vspace{0.1cm} $\sigmah_{j+n_j+1}^{\L}(3)$ & \vspace{0.1cm} $(-1)^{n_j+1}\overline{\omega}^{n_j}+1$ \\  \cline{1-1} \cline{3-4}
				\multirow{3}{2.7cm}{$G_2$} & & \vspace{0.1cm} $\sigmah_{j+n_j+1}^{\L}(4)$ & \vspace{0.1cm} $(-1)^{n_j+2}\overline{\omega}^{n_j+1}+1$ \\ \cline{2-4}
				& \multirow{2}{1.5cm}{$\sigmah_{j+n_j+2}^{\star}$} &  \vspace{0.1cm} $\sigmah_{j+n_j+2}^{\star}(4)$ & \vspace{0.1cm} $(-1)^{n_j+2}\overline{\omega}^{n_j+1}+1$ \\  \cline{3-4}
				& &  \vspace{0.1cm} $\sigmah_{j+n_j+2}^{\star}(1)$ & \vspace{0.1cm} $(-1)^{n_j+3}\overline{\omega}^{n_j+2}+1$ \\ \cline{2-4}
				\hline
				\hline
		\end{tabular}}
		\caption{The list of vertices of some of the $3$--cells around $e_j$ that are not identified with the endpoints of $e_j$. We recall that $\omega= -\frac{1}{2}\left( 1 + i \sqrt{3}\right)$.}
		\label{table:vertices_around_edge_L_1}
	\end{table}
	
	Similarly, if we travel clockwise around $e_j$, we have
	$$
	\sigma_j^{\L} \ \sigma_{j+1}^{\R} \cdots  \ \sigma_{j+n_j}^{\R} \ \sigma_{j+n_j+1}^{\L} \ \sigma_{j+n_j+2}^{\star}.
	$$
	The vertices of these cells that are not identified with the endpoints of $e_j$ are summarised in Table~\ref{table:vertices_around_edge_L_2}.
	
	\begin{table}[!t]
		{\centering
			\begin{tabular}{|| >{\centering\arraybackslash} m{2.5cm} | >{\centering\arraybackslash} m{3cm} | >{\centering\arraybackslash} m{4cm} | >{\centering\arraybackslash} m{3cm} ||} 
				\hline
				CR face pairing & $3$--cells & Vertices disjoint from $e_j$ & $\CC$--coordinates \\
				\hline\hline
				\multirow{3}{2.7cm}{$G_2$} & \vspace{0.1cm} \multirow{2}{1.5cm}{ $\sigmah_{j}^{\L}$}  &   \vspace{0.1cm} $\sigmah_j^{\L}(3)$ & \vspace{0.1cm} $0$ \\  \cline{3-4}
				& &\vspace{0.1cm} $\sigmah_j^{\L}(2)$ & \vspace{0.1cm} $-\omega$ \\  \cline{2-4}
				& \vspace{0.1cm} $\sigmah_{j+k}^{\R}$ & \vspace{0.1cm} $\sigmah_{j+k}^{\R}(2)$ & \vspace{0.1cm} $(-1)^{k+1} \overline{\omega}^{k}+1$ \\ \cline{1-1} \cline{3-4}
				\multirow{2}{2.7cm}{$G_3$} & $k=1,\dots,n_j$ & \vspace{0.1cm} $\sigmah_{j+k}^{\R}(1)$ & \vspace{0.1cm} $(-1)^{k+2} \overline{\omega}^{k+1}+1$ \\ \cline{2-4}
				& \multirow{2}{1.5cm}{$\sigmah_{j+n_j+1}^{\L}$} &  \vspace{0.1cm} $\sigmah_{j+n_j+1}^{\L}(2)$ & \vspace{0.1cm} $(-1)^{n_j+2} \overline{\omega}^{n_j+3}+1$ \\  \cline{1-1} \cline{3-4}
				\multirow{3}{2.7cm}{$G_1$} & & \vspace{0.1cm} $\sigmah_{j+n_j+1}^{\L}(1)$ & \vspace{0.1cm} $(-1)^{n_j+3} \overline{\omega}^{n_j+2}+1$ \\ \cline{2-4}
				& \multirow{2}{1.5cm}{$\sigmah_{j+n_j+2}^{\star}$} &  \vspace{0.1cm} $\sigmah_{j+n_j+2}^{\star}(1)$ & \vspace{0.1cm} $(-1)^{n_j+3} \overline{\omega}^{n_j+2}+1$ \\  \cline{3-4}
				& &  \vspace{0.1cm} $\sigmah_{j+n_j+2}^{\star}(4)$ & \vspace{0.1cm} $(-1)^{n_j+2} \overline{\omega}^{n_j+1}+1$ \\ \cline{2-4}
				\hline
				\hline
		\end{tabular}}
		\caption{The list of vertices of the remaining $3$--cells around $e_j$ that are not identified with the endpoints of $e_j$.}
		\label{table:vertices_around_edge_L_2}
	\end{table}
	
	\enlargethispage{\baselineskip}
	
	We remark that for all $k=1,\dots,n_j$, the $3$-cells $\sigmah_{j+k}^{\R}$ and $\sigmah_{j+k}^S$ cover a total angle of $2\pi$ around $e_j$. When $n_j=0$, the total angle around $e_j$ is exactly of $2 \pi$, hence in the general case the ramification order around $e_j$ is $n_j+1$.
\end{proof}

%%%%%%%%%%%%%%%%%%%%%%%%%%

\hfill

\subsection{The holonomy representation}
\label{subsec:holonomy}

It was mentioned in \S\ref{sec:CR_structure_bundles} that every branched CR structure admits a pair of a developing map and holonomy representation, defined up to the action of $\PU(2,1)$. Let $(\dev_f,\hol_f)$ be a representative pair associated to the branched CR structure on $M_f$. Here we summarise few facts about $\hol_f$, referring the reader to the author's PhD thesis for more details and the connection to the work of Fock and Goncharov on positive representations~\cite{Fock-Goncharov-2006}.

The fundamental group of $M_f$ is an HNN extension of the fundamental group of the base once--punctured torus $\ptorus$, namely the free group in two generators $\langle \alpha, \beta \rangle$. It has a standard presentation
$$
\pi_1(M_f) = \langle \alpha,\beta,\tau \ | \ \tau \alpha \tau^{-1} = f_*(\alpha), \ \tau \beta \tau^{-1} = f_*(\beta) \rangle,
$$ 
where $f_* : \langle \alpha, \beta \rangle \rightarrow \langle \alpha, \beta \rangle$ is the automorphism induced by $f$, and $\tau$ is represented by the base circle of the fiber bundle. If $M_f$ has flip sequence $\bw_f = \R^{a_0} \L^{b_0} \dots \R^{a_k} \L^{b_k} \R^c $ (the other case being similar), there is a choice of the class representative $(\dev_f,\hol_f)$ such that
$$
\hol_f(\alpha) = G_4^{-1} G_3, \quad \hol_f(\beta) = G_1^{-1} G_2, \quad \text{ and } \quad \hol_f(\tau) = G_4^{-a_0-c} G_1^{-b_0} \dots G_4^{-a_k} G_1^{-b_k}.
$$
Let $\rho : \pi_1(\ptorus) \rightarrow \PU(2,1)$ be the representation obtained by restricting $\hol_f$ to $\langle \alpha,\beta\rangle$. Then $\rho$ does not depend on $f$, namely it is a representation of $\pi_1(\ptorus)$ that always extends to a representation of $\pi_1(M_f)$. It is irreducible, but not strongly irreducible. Moreover, it is not faithful but it has infinite discrete image. In fact, its image $\rho(\pi_1(\ptorus))$ is a subgroup of the Eisenstein-Picard modular group $\PU(2,1,\ZZ[\omega])$, the subgroup of $\PU(2,1)$ with entries in the set of Eisenstein integers $\ZZ[\omega]$.

The representation $\rho$ was proved to have the above special properties while studying Fock and Goncharov's parametrisation of $\Charx(\ptorus)$, the decorated $\PGL(3,\CC)$--character variety of $\ptorus$. Using the inclusion map $\Hb \hookrightarrow \CP^2 $ together with its first complex jet, one induces a decoration on $\rho$, making its $\PGL(3,\CC)$ conjugacy class $[\rho]$ an element of $\Charx(\ptorus)$. Under this construction, the Fock-Goncharov coordinate of $[\rho]$ is
$$
P = \left( \omega,\omega,\omega,\omega,\omega,\omega,\omega,\omega \right), \quad 
\text{ where } \quad \omega = -\frac{1}{2}\left( 1 + i \sqrt{3}\right).
$$
The point $P$ and its complex conjugate are the only points in Fock-Goncharov moduli space that are fixed by every Fock-Goncharov edge flip.

%%%%%%%%%%%%%%%%%%%%%%%%%%

\makeatletter
\setlength{\@fptop}{0pt}
\makeatother

%%%%%%%%%%%%%%%%%%%%%%%%%%

\subsection*{Acknowledgements}
The material presented here is based on the PhD dissertation of the author, completed at the University of Sydney in $2018$. We heartily thank Stephan Tillmann, the author's PhD supervisor, for introducing him to the topic, the multiple stimulating conversations and unlimited advices. We thank the PhD examiners Jeffrey Danciger, Antonin Guilloux and Joan Porti whose comments helped improve this manuscript. We also thank Sam Ballas and Lorenzo Ruffoni for their useful suggestions on the original draft. Lastly, we acknowledge the support by the Commonwealth of Australia during the author's PhD.

%%%%%%%%%%
% References
%%%%%%%%%%

\bibliographystyle{plain}
\bibliography{biblio}
\addcontentsline{toc}{chapter}{References}

%%%%%%%%%%%%%%%%%%%%%%%%%%%%

\address{Alex Casella,\\ Departement of Mathematics,\\ Florida State University,\\ FL 32303 USA\\
(acasella@fsu.edu)\\https://www.math.fsu.edu/~casella/}

\Addresses

%%%%%%%%%%%%%%%%%%%%%%%%%%%%

\end{document}